\documentclass[11pt]{amsart}

\usepackage{amsmath,amssymb,amsfonts,amsthm}
\usepackage[margin=1in]{geometry}
\usepackage[hidelinks]{hyperref}
\usepackage[capitalize,noabbrev]{cleveref}
\usepackage{microtype}

\crefname{thm}{Theorem}{Theorems}
\crefname{prop}{Proposition}{Propositions}
\crefname{lem}{Lemma}{Lemmas}
\crefname{defn}{Definition}{Definitions}
\crefname{rem}{Remark}{Remarks}

\newtheorem{thm}{Theorem}[section]
\newtheorem{prop}[thm]{Proposition}
\newtheorem{lem}[thm]{Lemma}

\theoremstyle{definition}
\newtheorem{defn}[thm]{Definition}
\theoremstyle{remark}
\newtheorem{rem}[thm]{Remark}

\newcommand{\R}{\mathbb{R}}

\newcommand{\Z}{\mathbb{Z}}
\newcommand{\N}{\mathbb{N}}
\newcommand{\F}{\mathbb{F}}
\newcommand{\T}{\mathbb{T}}

\newcommand{\C}{\mathbb{C}}

\newcommand{\NN}{\mathcal{N}}

\newcommand{\al}{\alpha}
\newcommand{\be}{\beta}
\newcommand{\la}{\lambda}

\newcommand{\La}{\Lambda}
\newcommand{\de}{\delta}

\newcommand{\De}{\Delta}
\newcommand{\ga}{\gamma}

\newcommand{\om}{\omega}

\newcommand{\eps}{\varepsilon}
\newcommand{\si}{\sigma}

\newcommand{\1}{\mathbf{1}}

\newcommand{\subs}{\subseteq}

\newcommand{\abs}[1]{\lvert #1 \rvert}
\newcommand{\norm}[1]{\lVert #1 \rVert}

\newcommand{\cB}{\mathcal B}
\newcommand{\cD}{\mathcal D}
\newcommand{\cE}{\mathcal E}
\newcommand{\cK}{\mathcal K}

\newcommand{\cM}{\mathcal M}
\newcommand{\cN}{\mathcal N}
\newcommand{\cC}{\mathcal C}

\newcommand{\rank}{\operatorname{rank}}
\newcommand{\dist}{\operatorname{dist}}

\numberwithin{equation}{section}

\title[Simplex--center configurations]{Simplex--center configurations in dense subsets of
	Euclidean spaces and the integer lattice}

\author{\'Akos Magyar}

\thanks{The author was
	partially supported from Simons Foundation MPS-TSM-854813 and NFKIH Excellence 15421 grants.}

\begin{document}
	
	\begin{abstract}
		We obtain density Ramsey theorems for configurations consisting of the vertices of a simplex $\De_o$ together with their barycenter.  We prove that any subset $A\subs\R^n$ of positive upper density contains an isometric copy of all sufficiently large dilates of $\De_o$ together with its barycenter. As this configuration is non-spherical such results are not possible with respect to the quadratic Euclidean metric, we consider general metrics $\rho$ defined by a positive-definite, homogeneous forms of even degree at least four. We prove the analogous result in the discrete setting, for subsets $A$ of the integer lattice $\Z^n$, under some natural and necessary congruence restrictions on the scales $\la$ at which the set $A$ can contain an isometric copy of the simplex. 
	\end{abstract}
	
	\maketitle
	
	\section{Introduction.} In the 1970s Erd\H{o}s et al. \cite{EGRSS} initiated the study of geometric or Euclidean Ramsey theory. They studied finite point configurations $F$ such that for every $r$-coloring of a sufficiently high-dimensional Euclidean space there is at least one color class which contains an isometric copy of $F$. Later Bourgain \cite{Bourgain86} obtained a natural density analogue for non-degenerate simplices $\De=\{w_1,\ldots,w_m\}$. He showed that any positive density set $A\subs \R^m$ necessarily contains an isometric copy of its dilates $\la\De$, for sufficiently large scales $\la\geq \la(A,\De)$. An example of Graham \cite{Graham94} shows that in both the coloring and the density case positive results are only possible when the configuration is spherical, i.e. can be inscribed in a sphere. 
	
	However, it was shown in \cite{CMP17} that such results are still possible for certain non-spherical configurations if one changes the underlying metric of the ambient space. In particular it was shown that a positive density set $A\subs\R^n$ contains a three-term progression $x,x+y,x+2y$ with $|y|_p=\la$ for all $\la\geq \la(A)$, for $1<p<\infty,\ p\neq 2$, at least in large dimensions $n\geq n(p)$. Here $|y|_p=(\sum_{i=1}^n |y_i|^p)^{1/p}$ denotes the $\ell^p$-metric of the ambient space. Our first result is an extension to configurations consisting of the vertices of a simplex with $m$ vertices $\De=\{w_1,\ldots,w_m\}$ together with its barycenter with respect to metrics generated by certain positive, convex, homogeneous forms. To be more precise, let 
	\begin{equation}\label{eq:centered-simplex}
		\De_o=(w_1,\ldots,w_m),\qquad \sum_{i=1}^m w_i=0,
	\end{equation}
	be a fixed simplex in $\R^n$ centered at the origin. We will denote by 
	\begin{equation}\label{eq:E}
		E:=\binom m2,
	\end{equation}
	the number of its edges. 
	
	Let $Q\in\Z[x_1,\ldots,x_n]$ be homogeneous of even degree $r\geq 4$.
	Throughout we assume that $Q$ is positive-definite and
	convex, so that
	\[
	\rho(x):=Q(x)^{1/r}
	\]
	is a norm, thus it defines a metric on the underlying Euclidean space $\R^n$.
	We will also assume that $Q$ is non-singular in the sense that
	\begin{equation}\label{eq:Q-nonsingular}
		\nabla Q(x)=0\quad (x\in\C^n)
		\qquad\Longrightarrow\qquad x=0.
	\end{equation}
	
	We define
	\begin{equation}\label{eq:edge-data}
		\ell'_{ij}:=\rho(w_i-w_j)\quad \text{and}\quad \ell_{ij}=Q(w_i-w_j)\qquad (1\leq i<j\leq m).
	\end{equation}
	Note that $\ell'_{ij}=\ell_{ij}^{1/r}$ is the length of the edge of the simplex connecting the points $w_i$ and $w_j$; however it will be more convenient to work with the quantities $\ell_{ij}$.
	
	We refer to a configuration 
	\begin{equation}\label{eq:pattern}
		z,\quad z+v_1,\ldots,z+v_m,\qquad \sum_{i=1}^m v_i=0,
	\end{equation}
	as a simplex--center pattern, note that $z$ is the barycenter of the other $m$ points.  Prescribing
	\begin{equation}\label{eq:edge-equations}
		Q(v_i-v_j)=\la^r\ell_{ij}\qquad(i<j)
	\end{equation}
	makes the simplex $\De(z)=\{z+v_1,\ldots,z+v_m\}$ $\rho$-isometric to $\la\De_o$. This pattern is genuinely non-spherical in the ordinary Euclidean sense as the barycenter $z$ lies in the interior of the convex hull of the vertices.
	
	Write
	\begin{equation}\label{eq:H}
		H_m^n:=\left\{(v_1,\ldots,v_m)\in(\R^n)^m:
		\sum_{i=1}^m v_i=0\right\},
		\qquad D:=(m-1)n,
	\end{equation}
	and introduce the edge map
	\begin{equation}\label{eq:Phi}
		\Phi:H_m^n\longrightarrow\R^E,\qquad
		\Phi(v):=\bigl(Q(v_i-v_j)\bigr)_{i<j}.
	\end{equation}
	We define the notion of a $Q$-regular simplex,
	\begin{defn}[Regular simplex]\label{def:regular}
		The centered simplex $\De$ is \emph{$Q$-regular} if
		\begin{equation}\label{eq:regular}
			\rank\, (Jac\,\Phi(\De))=E.
		\end{equation}
	\end{defn}
	This means that the map $\Phi$ has maximal rank at the simplex $\De=(w_1,\ldots,w_m)\in (\R^n)^m$. Bourgain's theorem \cite{Bourgain86} treats simplices which are non-degenerate in the sense that their vertices are affinely independent. For the standard Euclidean metric and for triangles this rank condition is equivalent to the affine independency of the vertices, but for metrics arising from higher degree forms the latter is not sufficient when $m\geq4$, see \cref{sec:regularity}.
	
	\subsection{The Euclidean setting.} Recall the upper Banach density of a measurable set $A\subset\R^n$
	\[
	\overline d(A):=\limsup_{N\to\infty}\sup_{x\in\R^n}
	\frac{\abs{A\cap(x+[0,N]^n)}}{N^n}.
	\]
	Our first result is the following.
	\begin{thm}\label{thm:euclidean}
		Let $Q$ satisfy the preceding hypotheses and let $\De_o$ be an affinely
		independent, $Q$-regular simplex.  Suppose
		\begin{equation}\label{eq:euclidean-dimension}
			n>E(r-1)2^{r-1}.
		\end{equation}
		If $A\subset\R^n$ is measurable and $\overline d(A)>0$, there is
		$\la_0=\la_0(A,Q,\De_o)$ such that, for every $\la\geq\la_0$, the set $A$
		contains a $\rho$-isometric copy of $\la\De_o$ together with its barycenter, i.e. a  configuration of the form \eqref{eq:pattern} satisfying \eqref{eq:edge-equations}.
	\end{thm}
	
	Note that the exponential dependence of the dimension $n$ on the degree $r$ is a consequence of the generality of the metric $\rho=Q^{1/r}$, for diagonal forms e.g. for the $\ell^r$-metric the dimension depending only quadratically on $r$ using the best available results on the Waring problem \cite{Wooley}, however we will not pursue this special case in this note,
	
	
	The proof of Theorem \ref{thm:euclidean} follows broadly the approach of \cite{CMP17}. It is based on an oscillatory representation of Leray measure \cite{CLM21} of the configuration space and a decomposition of the range of integration into a small, an intermediate and a large region. Correspondingly the counting function of isometric copies of a scaled simplex in the set $A$ is decomposed into three parts; a main part, an intermediate part, and a high-frequency error term. The main term can be estimated below via standard results from additive combinatorics. The error term is controlled by the Gowers $U^3$-uniformity norm of the phase function $e^{itQ}$. The approximating Riemann sum of an appropriate power of a Gowers norm of the phase function $e^{itQ}$ is an exponential sum which can be estimated via Birch's Weyl-differencing method for non-singular forms \cite{Birch62}.  The intermediate terms are treated via estimates from time-frequency analysis, due to Muscalu, Tao and Thiele \cite{MTT02}.
	
	
	\subsection{The discrete setting.} Our main results concern subsets of the integer lattice $\Z^n$. We prove an analogue of Theorem \ref{thm:euclidean}, as well as a quantitative result about the existence of a single $\rho$-isometric copy in a finite $A\subs [N]^n$ of density $\de>0$. Both results can be viewed as extensions of the main results of \cite{MSimplex09}, from simplices to simplex-center configurations with respect to general metrics $\rho$, generated by positive-definite, non-singular forms. 
	
	Here we assume that $\De_o\subs \Z^n$ and $Q:\Z^n\to\Z$, we will refer to both the simplex $\De_o$ and the form $Q$ as being \emph{integral}. This assumption restricts the levels $L=\la^r$ to rational numbers with a fixed denominator depending only on $\De_o$. Indeed, a simplex $\De=\{v_1,\ldots,v_m\}$ is $\rho$-isometric to $\la\De_o$ if $Q(v_i-v_j)=L\ell_{ij}\in\N$ for all pairs $i<j$; we will only consider levels $L\in\N$. Another restriction on the levels is due to the existence of grids. Fixing a density $\de>0$ the grid $A=(q\Z)^n$ has density at least $\de>0$ if $q\leq \de^{-1/n}$ as $q^r\mid Q(v_i-v_j)$ for $v_i,v_j\in A$. Thus a discrete analogue of Theorem \ref{thm:euclidean} can only hold for those levels $L$ which are integer multiples of $q_0(\de):=\operatorname{lcm}\{1\leq q\leq \de^{-1/n}\}^r$, as it have to be valid for all such grids. 
	
	In order to find integer solutions $v=(v_1,\ldots,v_m)\in(\Z^n)^m$ of the diophantine system $Q(v_i-v_j)=L \ell_{ij},\ (1\leq i<j\leq m)$ one needs to make some additional local assumptions on the map $\Phi$. Moreover, have solutions $v=(v_1,\ldots,v_m)\in A^m$ restricted to set $A\subs\Z^n$ of positive upper density, one needs to have an ``abundance" of solutions in $v\in (\Z^n)^m$. The standard analytic tool which can provide such information is the Hardy-Littlewood circle method. Indeed, if the rank of the map $\Phi$ is sufficiently large and it also satisfies certain local conditions then it provides an asymptotic formula for the number of integer solutions to the system $\Phi(v)=L\ell$ \cite{Birch62}, at least for sufficiently large scales $L$. 
	
	Recall that the Birch rank $\cB(\Phi)$ of the map $\Phi$ is defined by 
	\begin{equation}\label{eq:Birch-locus}
		\cB(\Phi):=nm-\dim_{\C}V_\Phi^*\quad\text{where}\quad V_\Phi^*:=\{v\in (\C^n)^m:\rank\,(Jac\, \Phi(v))<E\}
	\end{equation}
	is the so-called singular locus of the map $\Phi$. The following explicit estimate provides a lower bound for the Birch rank of the system $\Phi$, exploiting the non-singularity of the form $Q$.
	
	\begin{prop}[Rank inherited from $Q$]\label{prop:system-rank}
		If $Q$ is non-singular, then
		\begin{equation}\label{eq:rank-lower}
			\cB(\Phi)\geq n-E+1.
		\end{equation}
	\end{prop}
	
	The proof of Proposition \ref{prop:system-rank} is will be given Section 4. As a corollary; in dimensions $n$ satisfying 
	\begin{equation}\label{eq:lattice-dimension}
		n-E\geq E(E+1)(r-1)2^{r-1},
	\end{equation}
	one has the lower bound
	\begin{equation}\label{eq:birch-threshold}
		\cB(\Phi)>E(E+1)(r-1)2^{r-1}.
	\end{equation}
	This rank condition allows the use of the Birch's circle method for the diophantine system $\Phi(v)=L\ell$ which is crucial for our main result in the discrete settings. 
	
	One also needs some local conditions satisfied by the system $\Phi$. Let $\La=H_m^n\cap \Z^{nm}$. For a prime $p$, define 
	\begin{equation}\label{eq:local-density}
		\sigma_p(L):=\lim_{s\to\infty}p^{-s(D-E)}
		\#\left\{v\in \La/p^s\La:\ 
		\Phi(v)\equiv L\ell\pmod{p^s}\right\}
	\end{equation}
	when the limit exists, and let $\mathfrak{S}(L):=\prod_{p\ prime}\si_p(L)$.
	
	Under \eqref{eq:birch-threshold}, the standard complete exponential-sum estimates used in Birch’s method give $\si_p(L)=1+o(p^{-2})$, at all sufficiently large primes \cite{Birch62} . At each of the finitely many remaining primes, Hensel lifting from a non-singular zero gives uniform positive lower bounds once$L$ is divisible by a sufficiently large power of $p$. Together with the Chinese remainder theorem, these facts yield the following standard lemma.
	
	\begin{lem}\label{lem:Q-regularity} Let $Q:\Z^n\to\Z$ be a positive-definite, non-singular form and let $\De$ be a $Q$-regular simplex. Assume that the dimension $n$ satisfies \eqref{eq:lattice-dimension}. If for all primes $p$ there exists $v^{(p)}\in {Q}_p^{nm}$ such that
		\begin{equation}\label{eq:local-nonsingular-zero}
			\Phi(v^{(p)})=0,\quad rank\,(Jac\,\Phi(v^{(p)}))=E,
		\end{equation}
		then one has 
		\begin{equation}\label{eq:uniform-zero}
			\beta_*:=\inf_{q\geq1}\beta_q(0)>0,
		\end{equation}
		for the quantities
		\begin{equation}\label{eq:zero-density}
			\beta_q(0):=q^{-(D-E)}\,
			\#\left\{v\in \La/q\La:
			\Phi(v)\equiv0\pmod q\right\},
		\end{equation}
		and there exist $q_o\in\N$ and $c,C>0$ such that
		\begin{equation}\label{eq:uniform-local}
			c\leq\mathfrak S(L)\leq C,\quad\text{whenever}\ q_o\mid L.
		\end{equation}
	\end{lem}

	\begin{rem}
		Let us remark that condition \eqref{eq:local-nonsingular-zero} automatically holds for all but finitely many primes. Indeed, for every sufficiently large prime of good reduction, the Birch-rank estimate gives  $p^{nm-E}(1+o(1))$ solutions of $\Phi(v)=0$, $v\in (\F_p^n)^m$, while the singular solutions form a lower-dimensional variety. Hence a non-singular solution $v^{(p)}$ exists modulo $p$. By a translation of the coordinates of $v^{(p)}$ by its center, one obtains a non-singular solution in $H_m^n(\mathbb Q{_p})$, which remains non-singular as the map $\Phi$ is invariant under such translations, at least for $p\nmid $. Then by Hensel's lemma it can be lifted to $(\Z_p^n)^m$, hence \eqref{eq:local-nonsingular-zero} needs to be verified only at finitely many exceptional primes.
	\end{rem}
	
	Our first result in the discrete settings is the following,
	\begin{thm}\label{thm:lattice-large scales}
		Let $Q:\Z^n\to\Z$ be a positive-definite, convex, non-singular form of degree $r>2$, such that the associated map $\Phi$ satisfies the local condition \eqref{eq:local-nonsingular-zero}, and let $\De_o$ be a $Q$-regular simplex. If
		\begin{equation}\label{eq:lattice-dimension-thm}
			n-E\geq E(E+1)(r-1)2^{r-1}.
		\end{equation}
		then the following holds for any set $A\subset\Z^n$ has positive upper Banach density $\delta>0$. 
		
		There are positive integers $q=q(\delta,Q,\De)$ and $L_0=L_0(A,Q,\De)$ such that, for every 
		\[
		L\geq L_0, \quad q\mid L,
		\]
		the set $A$ contains a barycentric configuration $\De(z)=\{z,z+v_1,\dots,z+v_m\}$ satisfying the equations
		\begin{equation}\label{eq:lattice-system}
			\sum_{i=1}^m v_i=0,\qquad
			Q(v_i-v_j)=L\ell_{ij}\quad(1\leq i<j\leq m).
		\end{equation}
	\end{thm}
	
	The main conclusion of Theorem \ref{thm:lattice-large scales} is that the set $A$ contains a simplex $\De$ together with its barycenter that is $\rho$-isometric to $\la\De_0$ for all $\la=L^{1/r}$, for all sufficiently large $L$ divisible by a fixed modulus $q$, which depends only on the initial data $\de,Q$ and $\De_o$. We will refer to such levels $L$, and corresponding and scales $\la=L^{1/r}$ as \emph{admissible} scales. 
	This compares to the main result of \cite{MSimplex09}, the crucial extra conclusion being that side vectors $v_1,\ldots,v_m$ satisfy the linear relation $v_1+\ldots+v_m=0$, and the isometry is defined with respect to metrics generated by positive-definite convex forms.
	
	The proof of Theorem \ref{thm:lattice-large scales} is based on a combination of elements of the the proof in the Euclidean case and arithmetic inputs from the Hardy-Littlewood circle method. If $A\subs [N]^n$ then the count of barycentric configurations $\De(z)\subs A$ isometric $\la \De_o$ is given by the quantity 
	\begin{equation}\label{eq:counting function}
		\NN_\la(A,\De_o):=\sum_{z\in \Z^n}\sum_{v\in H_m^n} f(z)f(z+v_1)\ldots f(z+v_m)\,\1_{\{\Phi(v)=L\ell\}}\end{equation}
	Our approach starts with the standard representation of the kernel
	\[
	\1_{\{\Phi(v)=L\ell\}}=\int_{\T^E} e (\al (\Phi(v)-L\ell)\,d\al,\qquad (e(\be):=e^{2\pi i\be}),
	\]
	and then a standard major/minor arcs decomposition of the phase variables $\al$ corresponding to the level $L$. The contribution of the minor arcs can be reduced after a change of variables, via three applications of the Cauchy-Schwarz inequality, to estimating the Gowers $U^3$ uniformity norm of the phase functions $e(\al Q(s))$, similarly as in the Euclidean setting. The $U^3$-norm is bounded by the $U^{r-1}$-norm which can be estimated via Weyl-differencing, already implicit in \cite{Birch62}. 
	
	The family of the majors arcs are enlarged to be centered at rational points with a fixed denominator $q=q(\de)$ and further decomposed into a small and an intermediate boxes. The integral over the small boxes provides the main term. The crucial difference from the Euclidean case that the corresponding kernel contains a local arithmetic factor $\chi(v)=q^E\,\1_{\{\Phi(v)\equiv 0\pmod{q}\}}$. The lower bound for the main term then is proved using the local condition \eqref{eq:uniform-zero} together with the finite multidimensional Szemer\'edi theorem. It is crucial that the lower bound is independent of the denominator $q$. This is different from the continuous case and does not provide effective quantitative lower bounds.
	
	Finally, the intermediate terms are handled by Poisson summation and Villarroya's transference theorem \cite{Villarroya11}, which provides a bridge between estimates for continuous multi-linear multipliers and their discrete analogues. This leads to an estimate for the total sum of the intermediate terms across a family of scales $\la_1\ll\ldots\ll\la_J$ by a quantity independent of the number of scales $J$; a standard indirect argument then finishes the proof.

	
	Our next result provides an effective bound on $N$ such that every $A\subs [N]^d$ of density $\de>0$ necessarily contains a barycentric configuration $\De(z)=\{z,z+v_1,\ldots,z+v_m\}$, such the simplex $\De=\{\{z,z+v_1,\ldots,z+v_m\}$ is $\rho$-isometric to $\la\De_o$ for some scale $\la>0$. We will refer to such a barycentric configuration $\De(z)$ $\rho$-\emph{similar copy} of $\De_o$. 
	
	\begin{thm}\label{thm:lattice-similar}
		Let $\de>0$ and let $\De_o\subs \Z^n$ be a fixed simplex.  Assume the form $Q$, the associated map $\Phi$, the simplex $\De_o$ and the dimension $n$ satisfy the hypotheses of Theorem \ref{thm:lattice-large scales}.
		
		Then there exists a constant $C=C(n,m,Q,\De_o)$ such that if $N\geq N(\de):=\exp\exp\,(C\,\de^{-C})$, then every set $A\subs [N]^n$ of size $|A|\geq \de N^n$ contains a barycentric configuration $\De(z)$ which is a $\rho$-similar copy of $\De_0$.
	\end{thm}
	
	Note that when the metric $\rho$ is the standard Euclidean metric generated by $Q(x)=|x|^2$, this is an
	ordinary Euclidean similar copy with dilation $\la$, and $z$ is the
	barycenter of its vertices.  For a general form $Q$, the terminology does not assert that a linear map carries $\la\De_o$ to $\De$.
	
	\section{Further barycentric configurations}
	
	The simplex-center configuration is a special case of general barycentric configurations $\De_{k,l}$ consisting of a simplex of $k$-points together with the barycenters of all of its $l$-faces. Much of the main ingredients of our proof extends to these configurations; a lower bound for the main term using a quantitative version of Szemer\'edi's theorem or its multi-dimensional extension and error term estimate via estimating Gowers norms. The missing ingredient is the applicability of a result in time frequency analysis that would provide a uniform bound on the sum of the intermediate terms over the scales $\la_j$ $(1\leq j\leq J)$, independent of $J$. However, it is plausible that good quantitative bounds are provided via the multi-scale orthogonality methods of Durcik-Kova\v{c} \cite{DK22} both in the Euclidean and discrete setting. We plan to address  general barycentric configurations in future work.
	
	\bigskip
	
	\section{Regularity with respect to the form $Q$.}
	\label{sec:regularity}
	
	Put
	\begin{equation}\label{eq:gij}
		g_{ij}:=\nabla Q(a_i-a_j),\qquad g_{ji}:=-g_{ij}.
	\end{equation}
	For $h=(h_1,\ldots,h_m)\in H_m^n$,
	\begin{equation}\label{eq:DPhi}
		\bigl(Jac\,\Phi(\De)h\bigr)_{ij}=g_{ij}\cdot(h_i-h_j).
	\end{equation}
	
	\begin{prop}[$Q$-regularity criterion]\label{prop:Q-test}
		The simplex $\De$ is $Q$-regular if and only if the only symmetric edge
		weights $\omega_{ij}=\omega_{ji}$ satisfying
		\begin{equation}\label{eq:stress}
			\sum_{j\ne i}\omega_{ij}g_{ij}=0
			\qquad(1\leq i\leq m)
		\end{equation}
		are $\omega_{ij}=0$.
	\end{prop}
	
	\begin{proof}
		Let $R_Q(\De)$ be the $E\times mn$ matrix whose $ij$-row has $g_{ij}$ in
		the $i$-block, $-g_{ij}$ in the $j$-block, and zero elsewhere.  A linear
		dependence among its rows is exactly \eqref{eq:stress}.  Simultaneous
		translations lie in the kernel of $D\Phi$, and every variation is the sum of a centered variation and a translation.  Restriction to $H_m^n$ therefore does not change the rank.  Hence
		\[
		\rank Jac\,\Phi(\De)=E
		\quad\Longleftrightarrow\quad
		\det\bigl(R_Q(\De)R_Q(\De)^T\bigr)\ne0
		\quad\Longleftrightarrow\quad
		\eqref{eq:stress}\text{ has only the zero solution}.
		\]
	\end{proof}
	
	\begin{prop}[Sufficient conditions]\label{prop:regular-cases}
		The following statements hold.
		\begin{enumerate}
			\item If $Q$ is a non-singular quadratic form, affine independence of the
			vertices implies $Q$-regularity.
			\item If $m=3$ and the unit sphere of $\rho=Q^{1/r}$ is smooth and
			strictly convex, every (n0n-degenerate) triangle is $Q$-regular.
			\item Regularity follows if the vertices can be ordered
			$i_1,\ldots,i_m$ so that, for every $s$, the vectors
			\begin{equation*}\label{eq:sequential-conormals}
				\{\nabla Q(a_{i_s}-a_{i_t}):t>s\}
			\end{equation*}
			are linearly independent.
			
		\end{enumerate}
	\end{prop}
	
	\begin{proof}
		For a quadratic form, $\nabla Q$ is an invertible linear map, and the usual simplex rigidity argument applies.  For a smooth strictly convex norm, the projective Gauss map is injective.  Thus the two conormals at a vertex of a non-collinear triangle are independent.
		
		Under \eqref{eq:sequential-conormals}, apply the equilibrium equation at
		$i_1$ to eliminate every edge from $i_1$ to a later vertex, and repeat at
		$i_2,i_3,\ldots$.  This eliminates every edge.  
		
	\end{proof}
	
	However, affine independence alone is not sufficient for higher simplices dimensional simplices, even for diagonal forms satisfying all the hypotheses on $Q$.
	
	\begin{prop}[A non-singular diagonal counterexample]\label{prop:counterexample}
		Let $r\geq4$ be even, put $p=r-1$ and $c=2^{-1/p}$, and take
		\[
		Q(x)=\sum_{\nu=1}^n b_\nu x_\nu^r,\qquad b_\nu>0.
		\]
		In the first three coordinates set
		\[
		a_0=0,\quad a_1=(1,-c,-c),\quad
		a_2=(-c,1,-c),\quad a_3=(-c,-c,1).
		\]
		These four vertices are affinely independent but not $Q$-regular.
		Translating all four vertices by their centroid gives an example in
		$H_4^n$.
	\end{prop}
	
	\begin{proof}
		The matrix with columns $a_1,a_2,a_3$ has eigenvalues
		$1+c,1+c,1-2c$, all nonzero.  Define
		\[
		\omega_{0i}=1\quad(1\leq i\leq3),\qquad
		\omega_{ij}=-\frac{1}{2(1+c)^p}\quad(1\leq i<j\leq3).
		\]
		Since $c^p=1/2$, the equilibrium equation at $a_0$ holds coordinatewise.
		At $a_1$, after suppressing the invertible factors $rb_\nu$, it is
		\[
		(1,-\tfrac12,-\tfrac12)
		-\frac12\bigl((1,-1,0)+(1,0,-1)\bigr)=0.
		\]
		At the other two vertices the same coordinate-wise identity holds after
		factoring out the corresponding nonzero coefficients $rb_\nu$.  
		Notice that $Q$ is non-singular, since
		$\nabla Q(x)=r(b_\nu x_\nu^p)_\nu$ vanishes only at the origin. 
	\end{proof}
	
	\section{Leray measure and the basic decomposition}
	\label{sec:decomposition}
	
	We will only use regularity near the fixed simplex $\De_o$ as the configuration space $S_{\Phi,v}:=\{v\in H_m^n:\ \Phi(v)=\la^r\ell\}$ may have singuarl point. Let $\chi_\De,\chi_0\in C_c^\infty(H_m^n)$ be smooth cut-off functions such that
	\begin{itemize}
		\item $\chi_{\De_o}$ is supported in a sufficiently small regular
		neighborhood of $\De_o$ and is positive at $\De_o$;
		\item $\chi_0$ is supported near the origin and is positive on a smaller
		neighborhood of the origin;
		\item $\operatorname{supp}\chi_0\cap\Phi^{-1}(\ell)=\emptyset$.
	\end{itemize}
	Set $\chi=\chi_\De+\chi_0$.  The normalized localized Leray measure maybe written formally as, see Section 7. 
	\begin{align}\label{eq:Leray}
		d\sigma_\la(v):&=
		\la^{rE-D}\chi(v/\la)
		\prod_{i<j}\delta\bigl(Q(v_i-v_j)-\la^r\ell_{ij}\bigr)\,dv_H\\
		&= \la^{rE-D}\chi(v/\la)\int_{\R^E} e\big(\sum_{i<j}\al_{ij}(Q(v_i-v_j)-\la^r\ell_{ij})\big)\,\prod_{i<j}d\al_{ij}
	\end{align}
	It is positive, nonzero, and supported near
	$\la\De_o$, while the $\chi_0$-part vanishes on the level set $S_{\Phi,v}$. Its normalization follows from the scaling properties,
	\begin{equation}\label{eq:Leray-scaling}
		dv_H(\la u)=\la^Ddu_H,\qquad
		\delta^{(E)}\bigl(\la^r(\Phi(u)-\ell)\bigr)
		=\la^{-rE}\delta^{(E)}(\Phi(u)-\ell).
	\end{equation}
	
	Choose an even function $\psi\in C_c^\infty(\R^E)$, $\psi\geq0$, with
	$\int\psi=1$, positive near both $0$ and $-\ell$.  Define
	\begin{align}
		\omega_\la(v)
		&:=\la^{-D}\chi(v/\la)
		\psi\bigl(\Phi(v)/\la^r-\ell\bigr),
		\label{eq:coarse-kernel}\\
		\omega_\la^\eps(v)
		&:=\la^{-D}\eps^{-E}\chi(v/\la)
		\psi\left(\frac{\Phi(v)/\la^r-\ell}{\eps}\right).
		\label{eq:fine-kernel}
	\end{align}
	The kernel $\omega_\la^\eps$ converges to $\sigma_\la$ in local coordinates as
	$\eps\to 0$.  Let
	\begin{equation}\label{eq:c-eps}
		c_\eps:=
		\frac{\int_{H_m^n}\omega_\la^\eps(v)\,dv_H}
		{\int_{H_m^n}\omega_\la(v)\,dv_H}.
	\end{equation}
	Scaling shows that $c_\eps$ is independent of $\la$, and it is easy to see  in local coordinates local coordinates  that it stays between two positive constants for small $\eps$.  Let
	\begin{equation}\label{eq:kernel-decomp}
		k_\la^\eps:=\omega_\la^\eps-c_\eps\omega_\la,\qquad
		e_\la^\eps:=\sigma_\la-\omega_\la^\eps.
	\end{equation}
	Then the correctly signed decomposition is
	\begin{equation}\label{eq:correct-decomp}
		\sigma_\la=c_\eps\omega_\la+k_\la^\eps+e_\la^\eps,
		\qquad \int k_\la^\eps=0.
	\end{equation}
	
	For $0\leq f\leq1$, supported in a cube of side $N$, define
	\begin{equation}\label{eq:counting-form}
		\cN_\la(f):=\int_{\R^n}\int_{H_m^n}
		f(z)\prod_{i=1}^m f(z+v_i)\,d\sigma_\la(v)\,dz.
	\end{equation}
	Thus
	\begin{equation}\label{eq:forms-decomp}
		\cN_\la(f)
		=c_\eps\cM_\la(f)+\cK_\la^\eps(f)+\cE_\la^\eps(f).
	\end{equation}
	
	\section{Proof of Theorem \ref{thm:euclidean}}\label{sec:Euclidean-proof} In this section we carry out the three main elements of the proof, a lower bound for the main term, en error estimate, and estimate for the total contribution of the intermediate terms for the decomposition \eqref{eq:forms-decomp}. We follow \cite{CMP17}, the new feature being an ``edge-separation" lemma which makes it possible to reduce the error estimate which to that of estimating $U^3$ uniformity norm of the phase function $e(tQ)$.
	
	\subsection{Lower bound for the main term.}
	
	\begin{lem}\label{lem:Varnavides}
		For every $m\geq2 $ and $\delta>0$, there is a constant $c=c_{m,n}(\de)>0$ with the following
		property.  If $A\subset[0,N]^n$ has measure at least $\delta N^n$ and
		$1\leq\la \leq \frac{c}{10n} N$, then
		\begin{equation}\label{eq:linear-count}
			\int_z\int_{\substack{v\in H_m^n\\|v_i|\leq c\la}}
			\1_A(z)\prod_{i=1}^m\1_A(z+v_i)\,dv_H\,dz
			\geq c\,N^n\la^D.
		\end{equation}
		Moreover, one may take $c=c_{m,n}(\de)=\exp\,(-C_{m,n} \log^7(2/\de))$ for $m\geq 3$ and\\
		$c_{m,n}(\de)=\exp\,(-C_{m,n} \log^9(2/\de))$ for $m=2$, with some constant $C_{m,n}>0.$
	\end{lem}
	
	\begin{proof} Partition the large cube into
		boxes of side comparable with $\la$, on at least $\de/2$ proportion of the boxes the relative density of the set $A$ is at least $\de/2$, let's call these boxes \emph{dense}. Discretize each dense box by a fine grid, apply the lower bound for the number of solutions to the linear, translation invariant equation
		\begin{equation}\label{eq:base-linear}
			x_1+\ldots+x_m-mz=0
		\end{equation}
		and average over translations, grid frames, and mesh sizes.  By the usual Varnavides averaging one has that a $c(\de)>0$ proportion of all solutions $(x_1,\ldots,x_m,z)$ to equation \eqref{eq:base-linear} are in $A^{m+1}$. Thickening the selected grid points inside Lebesgue-density cells and then averaging over the mesh offset removes the discretization, see \cite{CMP17}. Averaging over the dense boxes using the best current bounds \cite{Raghavan26,Kos25} $c_{m,n}(\de) = \exp\,(-C_{m,n} \log^{7} (2/\de))$ shows \eqref{eq:linear-count}.
	\end{proof}
	
	Since $\psi$ is positive near $-\ell$, the kernel $\omega_\la(v)$ is
	bounded below by a constant times $\la^{-D}$ when all $v_i$ are
	sufficiently small compared with $\la$.  Therefore
	\begin{equation}\label{eq:main-lower}
		\cM_\la(\1_A)\geq c(\delta,m,Q,\De)N^n
		\qquad(1\ll\la\ll_\delta N).
	\end{equation}
	
	\subsection{Pair isolation and the edge transform} In order to apply the Cauchy-Schwarz inequality to estimate the error term via the $U^3$-uniformity norm one needs to find an edge $(v_a-v_b)/2=s$ such that in the $s$-variable the phase 
	\begin{equation}\label{eq:Pt}
		P_t(v):=\sum_{i<j}t_{ij}Q(v_i-v_j).
	\end{equation}
	is non-vanishing for $t=(t_{ij})_{i<j}\in\R^E\backslash \{0\}$. 
	
	Let $1\leq a<b\leq m$. If all $v_j$, $j\neq a,b$, are fixed, the condition
	$\sum_i v_i=0$ fixes $v_a+v_b$.  Write
	\begin{equation}\label{eq:pair-coordinates}
		w:=-\sum_{j\ne a,b}v_j,\qquad
		v_a=\frac w2+s,\qquad v_b=\frac w2-s.
	\end{equation}
	Extend the variable $t=(t_{ij})_{i<j}$ symmetrically to $t=(t_{ij})_{i\ne j}$ via $t_{ji}:=t_{ij}$ if $i<j$. As a polynomial in $s$, the degree-$r$ part of $P_t$ is:
	\begin{equation}\label{eq:Cab-top}
		C_{ab}(t)Q(s),\qquad
		C_{ab}(t):=2^rt_{ab}
		+\sum_{j\ne a,b}(t_{aj}+t_{bj}).
	\end{equation}
	
	\begin{lem}\label{lem:pair-transform}
		We have,
		\begin{equation}\label{eq:Cab-lower}
			\max_{a<b}|C_{ab}(t)|\geq c_{m,r}|t|.
		\end{equation}
	\end{lem}
	
	\begin{proof} If we set $t_{ji}:=t_{ij}$ for $i<j$, then we may write
		\[
		P_t(v):=\frac{1}{2}\sum_{i\neq j}t_{ij}Q(v_i-v_j),
		\]
		Let's define $C_{ba}(t):=C_{ab}(t)$ where $C_{ab}(t)$ is defined by \eqref{eq:Cab-top}. Consider the linear map: $C_r:t\to (C_{ab}(t))_{a\neq b})$ mapping $\R^{m(m-1}$ to itself. We calculate the eigenvalues of this map. Assume $C_r(t)=\ga\, t$ for some eigenevalue $\ga$ and eigenvector $t\neq 0$. Write,
		\[
		\tau_a:=\sum_{b\neq a} t_{ab},\quad \tau:=\sum_a \tau_a
		\]
		Then, from \eqref{eq:Cab-top} we have for all $a\neq b$
		\[
		(2^r-2-\ga)t_{ab} +\tau_a+\tau_b=0.
		\]
		Summing for $b \neq a$ gives
		\[
		(2^r+m-4-\ga)\tau_a+\tau=0,
		\]
		for all $a\in [m]$. Summing for $a$, we get
		\[
		(2^r+2m-4-\ga)\tau=0.
		\]
		Then either $\ga=2^r+2m-4$ or $\tau=0$. If $\tau=0$ the either $\ga=2^r+m-4$ or $\tau_a=0$ for all $a\in [m]$. In the latter case either $\ga=2^r-2$ or $t_{ab}=0$ for all $a,b\in [m]$ and hence $t=0$. Thus the eigen-values are: $2^r-2,\ 2^r+m-4$ and $2^r+2m-4$. The operator $C_r$ is symmetric an its eigenvalues are at least $2^r-2$ thus $\|C_r(t)\|_2 \geq (2^r-2)\|t\|_2$. Hence
		\[
		\max_{a<b}|C_{ab}(t)|\geq c_{m,r}|t|,\quad\text{with}\quad c_{m,r}=\frac{2^r-2}{\sqrt{m(m-1)}}.
		\]
	\end{proof}
	
	If $g$ is compactly supported, and $B\subs \R^n$ is a box of size $P$, the normalized local norm of $g$ in the box $B$, is defined by
	\begin{equation}\label{eq:continuous-Gowers}
		\norm{g}_{U^s(B)}^{2^s}
		:=P^{-n(s+1)}\int_{\R^{n(s+1)}}\prod_{\omega\in\{0,1\}^s}
		\cC^{|\omega|}g(x+\omega\cdot h)\,\1_B(x+\om\cdot h)\,dx\,dh,
	\end{equation} where $\1_B$ is the indicator function of the box $B$. 
	
	\begin{lem}[Pair $U^3$ estimate]\label{lem:pair-U3}
		Let $F_0,F_1,F_2: B\to \C$ bounded in magnitude by 1, where $B$ is box of size $P$. 
		\[
		P^{-2n}\bigg|\int\int F_0(x)F_1(x+s)F_2(x-s)g(s)\,dx\,ds\,\bigg| \ll \|g\|_{U^3([-4P,4P]^n)},
		\]
		with an implicit constant depending only on the dimension $n$. 
	\end{lem}
	
	\begin{proof}
		Applying Cauchy--Schwarz inequality successively together with shifting the variables $x+s\to x$ after the first, and $x+2s\to x$ after the second application eliminates the factors $F_0,F_1,F_2$. After a third application of Cauchy--Schwarz the eighth power is
		bounded by a constant times the integral of
		\[
		P^{-4n}\int \prod_{\omega\in\{0,1\}^3}
		\cC^{|\omega|}g(s+\omega\cdot h)\,\1_{[-4P,4P]^n} (s+\om\cdot h)\,dh\,ds
		\]
		which is the eighth power localized $U^3([-4P,4P]^n)$-norm of the function $g$.
	\end{proof}
	
	\subsection{Gowers norm estimates for polynomial phase functions.} We will need to estimates the Gowers norms of certain phase functions. Such estimates are not directly available in the Euclidean settings for general polynomial phases $Q$, thus we first estimate the more standard discrete Gowers norms of phase functions $e(\theta Q(x))$.  
	Fortunately such estimates are known for phase functions $g(x)=e(\al Q(x)+R(x))$ where $deg\,R<r$ and $\al$ belongs to the so-called minor arcs, in fact if the form $Q$ has sufficiently large rank these are provide by provided by Birch \cite{Birch62}. To be more precise, let 
	\begin{equation}\label{eq:birch-parameters}
		s:=r-1,\qquad p:=2^{r-1},\qquad
		\kappa:=\frac{n}{(r-1)2^{r-1}}.
	\end{equation}
	If $g:\Z^n\to\mathbb C$ is supported in a box of side $O(P)$, define
	the normalized Gowers local norm by
	\begin{equation}\label{eq:discrete-local-Gowers}
		\norm{g}_{U_P^s}^{p}
		:=P^{-nr}\sum_{x,h_1,\ldots,h_s\in\Z^n}
		\prod_{\omega\in\{0,1\}^s}
		\cC^{\abs{\omega}}g(x+\omega\cdot h),
	\end{equation}
	where $g$ is extended by zero outside the box. 
	
	We recall Birch's minor arcs estimate for a single non-singular form of degree $r$.  For
	$0<\theta<1$, let recall the major/minor arcs at level $\theta$.
	\begin{equation}\label{eq:Birch-major-arcs}
		\mathfrak M_P(\theta)
		:=\bigcup_{1\le q\le P^{(r-1)\theta}}
		\ \bigcup_{\substack{a\in\Z\, (\mathrm{mod}\ q)\\(a,q)=1}}
		\left\{\alpha\in\T:
		2\abs{q\alpha-a}<P^{-r+(r-1)\theta}\right\},
		\qquad
		\mathfrak m_P(\theta):=\T\setminus\mathfrak M_P(\theta).
	\end{equation}
	
	\begin{lem}[Birch's minor-arc estimate for Gowers norms]
		\label{lem:Birch-minor-Gowers}
		Let $Q\in\Z[x_1,\ldots,x_n]$ be a homogeneous non-singular form of degree $r$ and let 
		$W\in C_c^\infty(\R^n)$.  If $R$ is any polynomial of degree
		less than $r$, then, uniformly in all coefficients of $R$, one has
		\begin{equation}\label{eq:minor-Gowers-theta}
			\norm{W(\cdot/P)e\bigl(\alpha Q+R_{<r}\bigr)}_{U_P^{r-1}}
			\ll_{Q,W}
			P^{-n\theta/2^{r-1}}
			\qquad\bigl(\alpha\in\mathfrak m_P(\theta)\bigr).
		\end{equation}
		Consequently, if $P^{-r}\le\abs{\alpha}<P^{-r/2}$, where $\alpha$ is
		represented in $[-1/2,1/2]$, then
		\begin{equation}\label{eq:minor-Gowers-small-alpha}
			\norm{W(\cdot/P)e\bigl(\alpha Q+R_{<r}\bigr)}_{U_P^{r-1}}
			\ll_{\eps,Q,W}
			\bigl(P^r\abs{\alpha}\bigr)^{-\kappa}.
		\end{equation}
	\end{lem}
	
	\begin{rem}\label{rem:smooth-pruning}
		Since all weights occurring here are smooth and compactly supported, the
		factor \(P^\eps\) that is appearing in in Lemma~4.4 may be omitted. Indeed, after the
		final differencing step, Poisson summation in the remaining variable
		replaces the usual reciprocal sums by a rapidly decreasing periodic
		kernel. The resulting dyadic sums are uniformly summable. Consequently,
		for every \(\kappa_0<\kappa\),
		\[
		\bigl\|
		W(\cdot/\lambda)e(\theta Q+R_{<r})
		\bigr\|_{U^{r-1}_\lambda}
		\ll_{\kappa_0,W}
		H^{-\kappa_0},
		\qquad
		\theta\notin\mathfrak M_1(\lambda;H),
		\]
		uniformly in the coefficients of \(R_{<r}\) and throughout the range of
		\(H\) used below.
	\end{rem}

	\begin{proof}
		Let $\Gamma_Q$ be the symmetric $r$-linear form associated with $Q$ and
		put
		\[
		\Psi_j(\mathbf h)
		:=(-1)^{r-1}r!\,\Gamma_Q(e_j,h_1,\ldots,h_{r-1}),
		\qquad \mathbf h=(h_1,\ldots,h_{r-1}).
		\]
		For
		\[
		g(x):=W(x/P)e\bigl(\alpha Q(x)+R_{<r}(x)\bigr),
		\]
		expansion of the $p$-th power in \eqref{eq:discrete-local-Gowers} gives
		\begin{equation}\label{eq:expanded-Gowers-power}
			\norm{g}_{U_P^{r-1}}^p
			=P^{-nr}\sum_{\mathbf h}\sum_x A_{\mathbf h}(x)
			e\left(\alpha\sum_{j=1}^n x_j\Psi_j(\mathbf h)
			+C_{\alpha,R}(\mathbf h)\right),
		\end{equation}
		where
		\[
		A_{\mathbf h}(x)
		:=\prod_{\omega\in\{0,1\}^{r-1}}
		\cC^{\abs{\omega}}
		W\left(\frac{x+\omega\cdot h}{P}\right).
		\]
		Here $\Delta_hF(x):=F(x)-F(x+h)$, and we used the fact that 
		\[
		\Delta_{h_1}\cdots\Delta_{h_{r-1}}
		\bigl(\alpha Q+R_{<r}\bigr)(x)
		=\alpha\sum_{j=1}^n x_j\Psi_j(\mathbf h)
		+C_{\alpha,R}(\mathbf h).
		\]
		In particular, all lower-degree terms are independent of the base variable
		$x$ and hence disappear after taking the absolute value of the inner sum.
		The support of $A_{\mathbf h}$ forces $\abs{h_i}\ll_W P$. The support of $A_{\mathbf h}$ is forces $|h_i|\ll P$ its rapidly decaying Fourier transform yields,
		\begin{equation}\label{eq:Birch-terminal-quantity}
			\norm{g}_{U_P^{r-1}}^p
			\ll_{W,C} \mathcal B_P(\alpha),
			\qquad
			\mathcal B_P(\alpha)
			:=P^{-nr}\sum_{\abs{\mathbf h}\ll_W P}
			\prod_{j=1}^n
			\big(1+P\norm{\alpha\Psi_j(\mathbf h)}_{\R/\Z}\big)^{-C}.
		\end{equation}
		
		The quantity $\mathcal B_P(\alpha)$ compares to the terminal quantity on the right of Birch's Lemma~2.1, equation~(2), 
		after division by $P^{np}$, with the difference that the quantity in Lemma~2.1 in \cite{Birch62} corresponds to $C=1$. The argument proving Birch's minor-arc estimate (Lemmas~2.2--3.3 and Lemma~4.3 of \cite{Birch62}), starting from
		this terminal quantity, applies without any changes. Indeed, the exponential sum
		enters that argument only through Lemma~2.1, which places its normalized
		$p$-th power below $\mathcal B_P(\alpha)$.  Starting instead from
		$\mathcal B_P(\alpha)^{1/p}$ gives the same alternative: either
		\begin{equation}\label{eq:Birch-terminal-minor-bound}
			\mathcal B_P(\alpha)^{1/p}
			\ll_\eps P^{-K\theta},
		\end{equation}
		or $\alpha\in\mathfrak M_P(\theta)$, or the forms $\Psi_i$ lie on Birch's
		singular locus.  Since $Q$ is non-singular, Birch's parameter $K$, defined
		by
		\[
		\dim V_Q^*=n-2^{r-1}K,
		\]
		equals $n/2^{r-1}$, and the singular alternative is absent.  Combining
		\eqref{eq:Birch-terminal-quantity} and
		\eqref{eq:Birch-terminal-minor-bound} proves
		\eqref{eq:minor-Gowers-theta}.
		
		Finally, Birch's ``small-$\alpha$" corollary following Lemma~4.3 applies
		to $\mathcal B_P(\alpha)^{1/p}$.  Indeed, choose $\theta_0$ such that 
		\[
		P^r\abs{\alpha}=P^{(r-1)\theta_0}
		\]
		and approach $\theta_)$ value of $\theta$ from below.  Since
		$P^{-r}\le\abs{\alpha}<P^{-r/2}$, one has
		$2(r-1)\theta<r$, and the minor arcs estimate \eqref{eq:Birch-terminal-minor-bound} applies. Thus
		\[
		P^{-K\theta}
		=\bigl(P^r\abs{\alpha}\bigr)^{-K/(r-1)}
		=\bigl(P^r\abs{\alpha}\bigr)^{-\kappa},
		\]
		which gives \eqref{eq:minor-Gowers-small-alpha}. 
	\end{proof}
	
	
	\begin{lem}[Transfer to Euclidean Gowers norms]
		\label{lem:Birch-Archimedean}
		Let $r\ge4$ and fix $\varphi\in C_c^\infty(\R^n)$.  Suppose
		\begin{equation}\label{eq:Arch-phase}
			P_t(y)=c_tQ(y)+R_t(y),\qquad \deg R_t<r,
		\end{equation}
		where, for $\abs{t}\ge1$,
		\[
		\abs{c_t}\asymp\abs{t}
		\quad\text{and the coefficients of }R_t\text{ are }O(1+\abs{t}).
		\]
		
		Let $W\in C_c\infty(\R^n)$ satisfying $\|W\|_{C^M}\leq C_W$.
		Then, 
		\begin{equation}\label{eq:Arch-Gowers-decay}
			\norm{W\, e(P_t)}_{U^3(\R^n)}
			\ll_{Q,W}(1+\abs{t})^{-\kappa},
		\end{equation}
		with implicit constant depending only on the support and finitely many derivatives of the weight function $W$.
	\end{lem}
	
	\begin{proof}
		The range $\abs{t}\le1$ is trivial.  Put $T:=2+\abs{t}$, choose
		\begin{equation}\label{eq:A-choice}
			A>1+\frac{n}{r-1},
			\qquad P:=\lceil T^A\rceil,
		\end{equation}
		and sample the continuous function on the mesh $P^{-1}\Z^n$.  Since the
		integrand defining the $p$-th power of the $U^{r-1}(\R^n)$-norm has fixed
		compact support and Lipschitz norm $O(T)$, the Riemann-sum comparison gives
		\begin{equation}\label{eq:Riemann-comparison}
			\norm{W e(P_t)}_{U^{r-1}(\R^n)}^p
			=\norm{W\,(\cdot/P)e\bigl(P_t(\cdot/P)\bigr)}_{U_P^{r-1}}^p
			+O(TP^{-1}).
		\end{equation}
		On the lattice,
		\[
		P_t(x/P)=\alpha Q(x)+R_{t,P}(x),
		\qquad \alpha:=\frac{c_t}{P^r},
		\qquad \deg R_{t,P}<r.
		\]
		Moreover,
		\[
		P^r\abs{\alpha}=\abs{c_t}\asymp T,
		\qquad P^{-r}\le\abs{\alpha}\le P^{-r/2}
		\]
		for $T$ sufficiently large.  Hence
		\eqref{eq:minor-Gowers-small-alpha} gives
		\begin{equation}\label{eq:lattice-minor-transfer}
			\norm{W(\cdot/P)e\bigl(P_t(\cdot/P)\bigr)}_{U_P^{r-1}}
			\ll_{Q,W}T^{-\kappa}.
		\end{equation}
		After taking the $p$-th root in \eqref{eq:Riemann-comparison}, the error is
		\[
		O\bigl((TP^{-1})^{1/p}\bigr)
		=O\bigl(T^{-(A-1)/2^{r-1}}\bigr)
		=o(T^{-\kappa})
		\]
		by \eqref{eq:A-choice}.  Combining this with
		\eqref{eq:lattice-minor-transfer} proves the corresponding
		$U^{r-1}(\R^n)$ estimate.  Since $r\ge4$, using monotonicity of the local Gowers norms gives $U^3\lesssim U^{r-1}$. This proves
		\eqref{eq:Arch-Gowers-decay}.
	\end{proof}
	
	\subsection{The high-frequency term} 
	Fourier inversion, followed by the changes of variables
	\(v=\lambda u\) and \(t=\lambda^r\alpha\), gives
	\[
	\mathcal E_\lambda^\varepsilon(f)
	=
	\int_{\mathbb R^E}
	\bigl(1-\widehat\psi(\varepsilon t)\bigr)
	e(-t\cdot\ell)\,
	\mathcal T_{\lambda,t}(f)\,dt,
	\]
	where
	\[
	\mathcal T_{\lambda,t}(f)
	:=
	\int_{\mathbb R^n}\int_{H_m^n}
	\chi(u)f(z)\prod_{i=1}^m f(z+\lambda u_i)\,
	e\,\bigl(P_t(u)\bigr)\,du_H\,dz
	\]
	and
	\[
	P_t(u):=t\cdot\Phi(u)
	=\sum_{i<j}t_{ij}Q(u_i-u_j).
	\]
	Fixing the remaining simplex variables and using \cref{lem:pair-transform}, choose
	$a<b$ with $|C_{ab}(t)|\gtrsim|t|$.  Then
	\cref{lem:pair-U3,lem:Birch-minor-Gowers} give
	\begin{equation}\label{eq:Tt-decay}
		|T_t(f)|\lesssim N^n(1+|t|)^{-\kappa}.
	\end{equation}
	Since $\widehat\psi(0)=1$ and $\psi$ is even,
	\begin{equation}\label{eq:rho-tail}
		|1-\widehat\psi(\eps t)|
		\lesssim\min(1,\eps^2|t|^2).
	\end{equation}
	Thus, whenever $\kappa>E$, there is $c>0$ such that
	\begin{equation}\label{eq:high-error}
		|\cE_\la^\eps(f)|\lesssim\eps^cN^n.
	\end{equation}
	One may take any $c<\min(2,\kappa-E)$.
	
	\subsection{The intermediate terms.}
	
	The kernels of the intermediate terms are of the form,
	\begin{equation}\label{eq:k-scaling}
		k_{\la_j}^\eps(v)=\la_j^{-D}k_1^\eps(v/\la_j),
		\qquad \int_{H_m^n}k_1^\eps=0,
	\end{equation}
	with $k_1^\eps$ smooth and compactly supported. We assume the scales $\la_j$ are \emph{lacunary} i.e. $\la_{j+1}\geq2\la_j$.
	
	\begin{prop}\label{prop:middle}
		If $\la_1<\cdots<\la_J$ is lacunary and $f$ is supported in a cube of
		side $N$, with $|f|\leq1$, then
		\begin{equation}\label{eq:middle-square}
			\sum_{j=1}^J|\cK_{\la_j}^\eps(f)|^2 \leq CN^{2n},
		\end{equation}
		with a constant $C=C(\eps,m,Q,\De)$ independent of $J$ and the sequence. 
	\end{prop}
	
	\begin{proof}
		The argument follows closely that of \cite[Section 5]{CMP17}. Applying Cauchy--Schwarz in the variable eliminates $f(z)$.  Expanding
		the square and summing in $j$ gives a $2m$-linear form with kernel
		\begin{equation}\label{eq:KJ}
			K_J^\eps(v,v')=\sum_{j=1}^J
			k_{\la_j}^\eps(v)\overline{k_{\la_j}^\eps(v')}.
		\end{equation}
		On the Fourier hyperplane
		\[
		\Gamma=\left\{(\xi_1,\ldots,\xi_m,\xi'_1,\ldots,\xi'_m):
		\sum_i(\xi_i+\xi'_i)=0\right\},
		\]
		the singular subspace is
		\begin{equation}\label{eq:Gamma-prime}
			\Gamma'=\{(\theta,\ldots,\theta,-\theta,\ldots,-\theta):
			\theta\in\R^n\}.
		\end{equation}
		To see this subspace directly, use $v_m=-\sum_{i<m}v_i$ as coordinates on
		$H_m^n$.  The kernel-frequency map is
		\begin{equation}\label{eq:kernel-frequency-map}
			L(\xi,\xi')=
			\bigl((\xi_i-\xi_m)_{i<m},(\xi'_i-\xi'_m)_{i<m}\bigr)
			\in (H_m^n)^*\times(H_m^n)^*.
		\end{equation}
		The multiplier is singular only when $L(\xi,\xi')=0$.  On $\Gamma$ this
		means that all unprimed frequencies equal some $\theta$, all primed
		frequencies equal some $\theta'$, and $m(\theta+\theta')=0$; hence
		$\theta'=-\theta$ and the singular set is exactly \eqref{eq:Gamma-prime}.
		It has integral rank one and is a graph over every frequency block.  In the
		notation of the integral-rank theorem there are $2m$ input blocks, each of
		dimension $n$, while the singular space has dimension $n$.  Thus its rank
		is $1<m=(2m)/2$, and the block-graph nondegeneracy is exactly the
		hypothesis in the rank-one integral case of
		\cite[Section~4.4]{DPT10}; see also \cite{CMP17,MTT02}.
		
		The symbol verification is short.  In kernel coordinates,
		\[
		m_J(\eta,\eta')=\sum_{j=1}^J
		\widehat{k_1^\eps}(\la_j\eta)
		\overline{\widehat{k_1^\eps}(\la_j\eta')}.
		\]
		Put $\varrho=|(\eta,\eta')|$.  For $\la_j\varrho\leq1$, the identity
		$\widehat{k_1^\eps}(0)=0$ controls the zeroth derivative, while the
		geometric series in $\la_j$ controls every positive derivative.  For
		$\la_j\varrho>1$, Schwartz decay gives the corresponding reverse
		geometric series.  Consequently, for every derivative required by the
		multiplier theorem,
		\begin{equation}\label{eq:MTT-symbol}
			|\partial^\beta m_J(\xi)|
			\leq C_{\beta,\eps}\dist(\xi,\Gamma')^{-|\beta|},
		\end{equation}
		uniformly in $J$.  Apply the theorem with all $2m$ input exponents equal to
		$2m$.  Their $L^{2m}$ norms contribute $N^n$, and the initial
		Cauchy--Schwarz factor contributes another $N^n$, proving
		\eqref{eq:middle-square}.
	\end{proof}
	
	\begin{proof}[Proof of \cref{thm:euclidean}]
		Suppose arbitrarily large scales are missing and select a finite lacunary
		sequence $\la_1<\cdots<\la_J$ of missing scales.  Upper Banach density
		provides a cube of side $N\gg\la_J$ in which $A$ has density bounded below by $\de$. Replace $A$ by its intersection with this cube and put $f=\1_A$.
		
		At every missing scale $\cN_{\la_j}(f)=0$.  Choose $\eps$ so small that
		\eqref{eq:high-error} is at most half the lower bound
		\eqref{eq:main-lower}.  Since $c_\eps\asymp1$,
		\eqref{eq:forms-decomp} gives
		\[
		|\cK_{\la_j}^\eps(f)|\geq c_{m,n}(\de) N^n
		\qquad(1\leq j\leq J).
		\]
		Therefore
		\[
		J\,c_{m,n}(\de)^2N^{2n}
		\leq\sum_{j=1}^J|\cK_{\la_j}^\eps(f)|^2
		\lesssim_\eps N^{2n},
		\]
		contradicting \cref{prop:middle} for large $J$.  
	\end{proof}
	
	\section{The discrete settings}\label{sec:lattice-proof} Here we describe the proofs of our main results in the setting of the integer lattice $\Z^n$.
	Fix a level $L$ and write
	\[
	P:=L^{1/r}
	\]
	for its physical scale.
	We insert the same cut-off function $\chi$ as in the Euclidean case, and define
	\begin{equation}\label{eq:discrete-kernel}
		\sigma_L^{\Z}(v):=
		P^{rE-D}\chi(v/P)
		\prod_{i<j}\1_{\{Q(v_i-v_j)=L\ell_{ij}\}},
		\qquad v\in H_{m,\Z}^n.
	\end{equation}
	Note that $\sigma_L^{\Z}$ is normalized if the system $\Phi$ has sufficiently large rank \cite{Birch62}, which we verify in \cref{subsec:rank-Phi}.
	
	For $f$ supported in a box of side $N$, the normalized lattice count is
	\begin{equation}\label{eq:discrete-count}
		\cN_L^{\Z}(f):=N^{-n}\sum_{z\in\Z^n}
		\sum_{v\in H_{m,\Z}^n}
		f(z)\prod_{i=1}^mf(z+v_i)\,\sigma_L^{\Z}(v).
	\end{equation}
	Fourier inversion gives the representation
	\begin{equation}\label{eq:circle-representation}
		\sigma_L^{\Z}(v)=P^{rE-D}\chi(v/P)
		\int_{\T^E}e\bigl(\alpha\cdot(\Phi(v)-L\ell)\bigr)\,d\alpha,
	\end{equation}
	
	\subsection{The rank of the map $\Phi$}
	\label{subsec:rank-Phi}
	We prove the lower bound for the rank of the system $\Phi$ by exploiting the non-singularity of the form $Q$. This will be needed to apply the circle method to the Diophantine system $\Phi(v)=L\ell$.
	
	\begin{proof}[Proof of \cref{prop:system-rank}]
		For $0\ne t\in\C^E$, let $P_t=t\cdot\Phi$.  By
		\cref{lem:pair-transform} we can choose $a<b$ with $C_{ab}(t)\ne0$ and use
		\eqref{eq:pair-coordinates}. Fixing all edges $v_i$ for $i\neq a,b$ and setting $s=(v_a-v_b)/2$ and $u=(v_i)_{i\neq a,b}$, we have
		\[
		P_t(v)=P_t(s,u)=C_{ab}(t)Q(s)+\text{terms of lower degree in }s.
		\]
		Let
		\[
		F_j^{u}(s):=\partial_{s_j}P_t(u,s),
		\qquad 1\leq j\leq n.
		\]
		These are polynomials of degree at most \(r-1\) in \(s\), and their
		homogeneous parts of degree \(r-1\) are
		\[
		C_{ab}(t)\partial_jQ(s),
		\qquad 1\leq j\leq n.
		\]
		Since \(Q\) is non-singular, the equations
		\[
		\partial_1Q(s)=\cdots=\partial_nQ(s)=0
		\]
		have no nonzero common solution in \(\mathbb C^n\).
		
		We claim that, for every fixed \(u\), the common zero set
		\[
		\{s\in\mathbb C^n:
		F_1^{u}(s)=\cdots=F_n^{u}(s)=0\}
		\]
		is finite. Indeed, homogenize the polynomials \(F_j^{u}\) in the
		variables \((s,w)\). On the hyperplane at infinity \(w=0\), the
		homogenized equations reduce to
		\[
		C_{ab}(t)\partial_jQ(s)=0,
		\qquad 1\leq j\leq n,
		\]
		which have no projective solution \([s:0]\), however positive-dimensional
		affine common zero set would have a projective closure meeting the
		hyperplane at infinity.
		
		Consequently, every fiber of the projection
		\[
		\{(u,s):\nabla_sP_t(u,s)=0\}\longrightarrow\mathbb C^{D-n},
		\qquad (u,s)\longmapsto u,
		\]
		is zero-dimensional. The fiber-dimension theorem therefore gives
		\[
		\dim\{(u,s):\nabla_sP_t(u,s)=0\}\leq D-n.
		\]
		Since
		\[
		V_{P_t}^{*}
		=\{(u,s):\nabla_{u,s}P_t(u,s)=0\}
		\subseteq
		\{(u,s):\nabla_sP_t(u,s)=0\},
		\]
		we conclude that
		\[
		\dim V_{P_t}^{*}\leq D-n.
		\]
		Consider the ``incidence" variety,
		\[
		\mathcal V_\Phi:=\{(v,[t])\in H_{m,\C}^n\times\mathbb P^{E-1}:
		\nabla_v(t\cdot\Phi)(v)=0\},
		\]
		where $[t]$ denotes the projective line through $t$. Every fiber over $[t]$ has dimension at most $D-n$, so $\dim\mathcal V_\Phi\leq D-n+E-1$.  Its projection to $H_{m,\C}^n$ is $V_\Phi^*:=\{v\in H_{m,\C}^n:\ rank\, Jac_\Phi (v)<E\}$. Indeed, if $v\in V_\Phi^*$ then there exists $t\in\C^E,\ t\ne 0$ such that $\nabla_v (t\cdot\Phi)=0$. Therefore
		$\dim V_\Phi^*\leq D-n+E-1$, which implies
		$\cB(\Phi)\geq n-E+1$.
	\end{proof}
	
	\subsection{The circle-method decomposition}
	
	Fix a physical scale $P$, or equivalently a level $L=P^r$. As usual in the circle method, we decompose the range of integration first into major/minor arcs depending on a parameter $\eps=\eps(\de)$. We say that $\al$ is in a major arc at height $R$ if there exists $1\leq q\leq R$ such that $\|q\al\|\leq R/L$, where $\|\be\|=|\{\be\}|_\infty$ and $\{\be\}\in [-\frac{1}{2},\frac{1}{2}]^E$ is the fractional part of $\be\in\R^E$. We denote the union of these major arcs by $\mathfrak M_{P,R}$ and use heights $\eps^{-1/2}\leq R\leq P$. We enlarge the major arcs at height $R=\eps^{-1/2}$ by taking $q=q(\eps):=\operatorname{lcm}\{1\leq q\leq \eps^{-1/2}\}$ and cubes of side comparable to $\eps^{-1/2}/L$ centered at $a/q\in\T^E$, with $a\in[q]^E$; denote their union by $\mathfrak M'_{\eps,P}$. We also assume throughout this section that $q\mid L$. Then we have the decomposition
	\begin{equation}\label{eq:discrete-decomp}
		\sigma_L^{\Z}
		=c_\eps\omega_{q,L}+k_{q,L}^\eps+e_{q,L}^\eps,
	\end{equation}
	where
	\[
	\Psi_{P,\tau}(\alpha)
	:=\sum_{k\in\Z^E}\widehat\psi\bigl(\tau P^r(\alpha-k)\bigr),
	\qquad \tau>0,
	\]
	is the periodization of the Euclidean multiplier.  Thus the following torus
	representations are exact:
	\begin{align}\label{5.6}
		\om_{q,L}(v)&:=P^{rE-D}\chi(v/P) \int_{\T^E} \sum_{a\in [q]^E} e\big(\al\cdot\,(\Phi(v)-L\ell)\big) \,\Psi_{P,1}(\al-a/q)\,d\al,\nonumber\\
		\om_{q,L}^\eps(v)&:= P^{rE-D}\chi(v/P) \int_{\T^E} \sum_{a\in [q]^E} e\big(\al\cdot\,(\Phi(v)-L\ell)\big) \,\Psi_{P,\eps}(\al-a/q)\,d\al.
	\end{align}
	Set
	\[
	k_{q,L}^\eps:=\omega_{q,L}^\eps-c_\eps\omega_{q,L},
	\qquad e_{q,L}^\eps:=\sigma_L^\Z-\omega_{q,L}^\eps.
	\]
	Changing variables $\al:=\al-a/q$ one obtains 
	\begin{equation}\label{5.7}
		\om_{q,L}(v)=\sum_{a\in [q]^E} e\bigg(\frac{a\cdot(\Phi(v)-L\ell)}{q}\bigg) \om_P(v) = q^E\,\1_{\{\Phi(v)\equiv 0 \pmod{q}\}}\,\om_P(v)=:\chi_q(v)\om_P(v),
	\end{equation}
	assuming $v\in H_{m,\Z}^n$ and $q=q(\eps)\mid L$. Thus the extra arithmetic factor in the discrete case is
	\begin{equation}\label{eq:chiq}
		\chi_q(v):=q^E\1_{\{\Phi(v)\equiv0\pmod q\}},
	\end{equation}
	By scaling, we have that
	\begin{equation}\label{eq:discrete-coarse-middle}
		\omega_{q,L}(v)=\chi_q(v)\omega_P(v),\qquad
		k_{q,L}^\eps(v)=\chi_q(v)k_P^\eps(v),
	\end{equation}
	$\om_P$ and $k_P^\eps$ being the kernels in \eqref{eq:kernel-decomp}.

	For $0\leq f\leq1$ supported in $[1,N]^n$, write
	\begin{equation}\label{eq:discrete-main-form}
		\cM_{q,L}^{\Z}(f):=N^{-n}\sum_{z\in\Z^n}
		\sum_{v\in H_{m,\Z}^n}
		f(z)\prod_{i=1}^mf(z+v_i)\omega_{q,L}(v),
	\end{equation}
	and define $\cK_{q,L}^{\eps,\Z}$ and
	$\cE_{q,L}^{\eps,\Z}$ analogously using $k_{q,L}^\eps$ and
	$e_{q,L}^\eps$.  Thus \eqref{eq:discrete-decomp} gives
	\begin{equation}\label{eq:discrete-forms-decomp}
		\cN_L^{\Z}(f)=c_\eps\cM_{q,L}^{\Z}(f)
		+\cK_{q,L}^{\eps,\Z}(f)+\cE_{q,L}^{\eps,\Z}(f).
	\end{equation}
	
	\medskip
	\begin{prop}[Main term lower bound.]\label{prop:disc-main}
		Assume \eqref{eq:uniform-zero}.  For every $\delta>0$ there are
		$c=c(\delta,m,n,Q,\De)>0$, $\eta=\eta(\delta,m,n,Q,\De)>0$, and
		$C=C(\delta,m,n,Q,\De)<\infty$ such that,
		if $A\subset[1,N]^n\cap\Z^n$ has density at least $\delta$, $q\mid L$,
		$P=L^{1/r}$, and
		\[
		Cq\leq P\leq\eta N,
		\]
		then
		\begin{equation}\label{eq:disc-main}
			\cM_{q,L}^{\Z}(\1_A)
			\geq c(\delta,m,n,Q,\De)>0,
		\end{equation}
	\end{prop}
	
	We remark that it is crucial in our argument that the lower bound is independent of $q$, $L$, and $N$.
	
	\begin{proof}
		Let $s=m-1$ and identify $H_{m,\Z}^n$ with $(\Z^n)^s$ by writing
		\begin{equation}\label{eq:centered-frame}
			R=(r_1,\ldots,r_s,r_m),\qquad
			r_m=-\sum_{j=1}^sr_j.
		\end{equation}
		Choose $\theta>0$ so small that
		\begin{equation}\label{eq:omega-origin-lower}
			\omega_P(v)\geq c_0\,P^{-D},
			\quad\text{whenever }\max_i|v_i|_\infty\leq\theta P.
		\end{equation}
		This is possible because $\chi_0$ is positive near the origin and $\psi$
		is positive near $-\ell$.
		
		We embed $A$ in the finite group
		\[
		\Gamma_N:=(\Z/(4N+1)\Z)^n.
		\]
		Its density in $\Gamma_N$ is at least
		$\delta_0:=5^{-n}\delta$. By the finite multidimensional Szemer\'edi theorem \cite{FK78} there exists $K=K(\delta_0,m)$ such that
		every subset of $[K]^s$ of density at least $\delta_0/2$ contains a
		nontrivial homothetic copy
		\begin{equation}\label{eq:coefficient-copy}
			a+tP_m\subset[K]^s,\qquad 1\leq t\leq K,
		\end{equation}
		where $P_m=\{0,e_1,\ldots,e_{m-1},-\sum_{i=1}^{m-1} e_i\}$.
		
		Let $\mathcal R_q(P)$ be the set of integral frames \eqref{eq:centered-frame}
		such that
		\begin{equation}\label{eq:small-null-frames}
			|r_j|_\infty\leq\frac{\theta P}{sK}\quad(1\leq j\leq s),
			\qquad \Phi(R)\equiv0\pmod q.
		\end{equation}
		Each frame $v\in H_{m,(Z/qZ)}^n$ such that $\Phi\equiv 0\pmod{q}$ has
		$\gg(P/(Kq))^D$ such lifts if $P\geq Cq$. Thus it follows from
		\eqref{eq:zero-density}--\eqref{eq:uniform-zero} that
		\begin{equation}\label{eq:small-frame-count}
			|\mathcal R_q(P)|
			\gtrsim_{\delta,m,n,Q,\De}\beta_*P^Dq^{-E}.
		\end{equation}
		
		Fix $R\in\mathcal R_q(P)$.  For $x\in\Gamma_N$, define the coefficient
		set
		\[
		B_{x,R}:=\left\{u\in[K]^s:
		x+\sum_{j=1}^su_jr_j\in A\right\}.
		\]
		Translation invariance of $\Gamma_N$ gives
		\[
		\mathbb E_{x\in\Gamma_N}\frac{|B_{x,R}|}{K^s}
		=\frac{|A|}{|\Gamma_N|}\geq\delta_0.
		\]
		Consequently, $\gg_{\delta,n}|\Gamma_N|$ choices of $x$ have
		$|B_{x,R}|\geq(\delta_0/2)K^s$.  For each such $x$, choose a copy
		\eqref{eq:coefficient-copy} and put
		\[
		z=x+\sum_{j=1}^sa_jr_j,
		\qquad v_j=tr_j\ (1\leq j\leq s),
		\qquad v_m=-t\sum_{j=1}^sr_j.
		\]
		Then $z,z+v_1,\ldots,z+v_m$ all belong to $A$.  Moreover,
		homogeneity gives
		\begin{equation}\label{eq:dilate-preserves-nullity}
			\Phi(v)=t^r\Phi(R)\equiv0\pmod q,
		\end{equation}
		and \eqref{eq:small-null-frames} gives
		$\max_i|v_i|_\infty\leq\theta P$.
		
		The map $(R,x,a,t)\mapsto(z,v)$ has multiplicity at most $K^{s+1}$:
		after $t$ and $a$ are fixed, $R=v/t$ and
		$x=z-\sum_ja_jr_j$ are fixed.  Taking $\eta$ sufficiently small that
		$\theta\eta<1$ also prevents wraparound in $\Gamma_N$. Thus every modular configuration counted is an actual integer configuration. Summing over
		$R\in\mathcal R_q(P)$ and using \eqref{eq:small-frame-count}, we obtain
		\begin{equation}\label{eq:unweighted-null-count}
			\sum_z\sum_{\substack{v\in H_{m,\Z}^n, |v|_\infty\leq\theta P\\
					\Phi(v)\equiv0\pmod q}}
			\1_A(z)\prod_{i=1}^m\1_A(z+v_i)
			\gtrsim_{\delta,m,n,Q,\De}\beta_*N^nP^Dq^{-E}.
		\end{equation}
		Finally, \eqref{eq:chiq}, \eqref{eq:omega-origin-lower}, and
		\eqref{eq:unweighted-null-count} give
		\[
		\cM_{q,L}^{\Z}(\1_A)
		\gtrsim N^{-n}q^EP^{-D}
		\bigl(\beta_*N^nP^Dq^{-E}\bigr)
		\gtrsim_{\delta,m,n,Q,\De}\beta_*>0,
		\]
		which proves \eqref{eq:disc-main}.
	\end{proof}
	
	For a smooth cutoff $W$ on $H_m^n$, introduce the normalized oscillatory
	form
	\[
	\Lambda_{\alpha,P}(f):=
	N^{-n}P^{-D}\sum_z\sum_{v\in H_{m,\Z}^n}
	W(v/P)f_0(z)\prod_{i=1}^mf_i(z+v_i)e(\alpha\cdot\Phi(v)).
	\]
	
	\begin{rem} We remark that for $m=2$, i.e. for three-term progressions $x-y,x,x+y$, one can use Roth's theorem, which provides a quantitative bound. This simplifies much of the discussion in Section~\ref{sec:quantitative-lattice} and leads to a quick double-exponential bound for \cref{thm:quantitative-lattice}. Thus it would be desirable to obtain a quantitative lower bound for $m\geq 3$ as well, bypassing the multidimensional Szemer\'edi theorem.
	\end{rem}

	\begin{lem}\label{lem:pruned-lattice}
		For \(M\geq2\), define
		\[
		\mathfrak M_C(P;M)
		:=
		\bigcup_{1\leq q\leq M}
		\left\{
		\alpha\in\T^E:
		\|q\,C_r\alpha\|_{\T^E,\infty}\leq MP^{-r}
		\right\}.
		\]
		Suppose that
		\[
		E(E+1)<\kappa_0<\kappa,\quad\text{with}\quad \kappa=\frac{n}{(r-1)2^{r-1}}.
		\]
		Let
		\[
		\gamma:=\frac{\kappa_0}{E}-(E+1)>0.
		\]
		Then, for every fixed \(M_0\geq2\),
		\begin{equation}\label{eq:integrated-minor}
			P^{rE}
			\int_{\T^E\setminus\mathfrak M_C(P;M_0)}
			|\Lambda_{\alpha,P}(f)|\,d\alpha
			\ll
			M_0^{-\gamma},
		\end{equation}
		uniformly for \(|f_i|\leq1\).
	\end{lem}
	
	\begin{proof}
		For \(H\geq2\), put
		\[
		\mathfrak M_1(P;H)
		:=
		\bigcup_{1\leq q\leq H}
		\left\{
		\theta\in\T:
		\|q\theta\|_{\T}\leq HP^{-r}
		\right\}.
		\]
		Taking \(H=P^{(r-1)\vartheta}\), these are, up to harmless
		constant changes, the major arcs \(\mathfrak M_P(\vartheta)\)
		of \cref{lem:Birch-minor-Gowers}.  The standard dyadically ``pruned" consequence of
		that lemma therefore gives, for every \(\kappa_0<\kappa\),
		\begin{equation}\label{eq:scalar-pruned}
			\left\|
			W_0(\cdot/P)e(\theta Q+R_{<r})
			\right\|_{U^{r-1}_P}
			\ll
			H^{-\kappa_0},\quad\text{for}\quad
			\theta\notin\mathfrak M_1(P;H),
		\end{equation}
		uniformly in the coefficients of \(R_{<r}\) and for the uniformly
		smooth family of weights \(W_0\) arising below.  The remainder in
		\eqref{eq:scalar-pruned} is uniform under the dyadic summation used
		below.
		
		Fix an edge \(a<b\) and all vertices other than \(v_a,v_b\).  Using
		\[
		w=-\sum_{j\neq a,b}v_j,\qquad
		v_b=-u,\qquad v_a=w+u,
		\]
		we have
		\[
		\alpha\cdot\Phi(v)
		=
		(C_r\alpha)_{ab}Q(u)+R_{\alpha,v'}(u),
		\qquad
		\deg R_{\alpha,v'}<r.
		\]
		The Cauchy--Schwarz argument of \cref{lem:pair-U3}, followed by
		\(U^3\lesssim U^{r-1}\), gives
		\begin{equation}\label{eq:discrete-pair-GVN}
			|\Lambda_{\alpha,P}(f)|
			\lesssim \mathbb E_{v'}
			\left\| W_{v'}(\cdot/P)
			e\bigl((C_r\alpha)_{ab}Q+R_{\alpha,v'}\bigr)
			\right\|_{U^{r-1}_P}.
		\end{equation}
		
		Suppose that every coordinate
		\(\beta_{ab}:=(C_r\alpha)_{ab}\) belongs to
		\(\mathfrak M_1(P;H)\).  Multiplying the corresponding \(E\)
		denominators gives a common denominator \(q\leq H^E\) satisfying
		\[
		\|q\,C_r\alpha\|_{\T^E}\leq H^EP^{-r}.
		\]
		Taking \(H=M^{1/E}\), we conclude that
		\[
		\alpha\notin\mathfrak M_C(P;M)
		\quad\Longrightarrow\quad
		(C_r\alpha)_{ab}\notin\mathfrak M_1(P;M^{1/E})
		\]
		for at least one edge \(a<b\).  Hence
		\eqref{eq:scalar-pruned}--\eqref{eq:discrete-pair-GVN} give
		\begin{equation}\label{eq:vector-minor-saving}
			|\Lambda_{\alpha,P}(f)|
			\ll
			M^{-\kappa_0/E},
			\qquad\text{for}\quad
			\alpha\notin\mathfrak M_C(P;M).
		\end{equation}
		
		Since \(C_r\) is an integral non-singular matrix, each map
		\(qC_r:\T^E\to\T^E\) is a surjective torus endomorphism.  Its
		pullback preserves normalized Haar measure, and consequently
		\begin{equation}\label{eq:vector-major-measure}
			|\mathfrak M_C(P;M)|
			\lesssim
			\sum_{q\leq M}(MP^{-r})^E
			\lesssim
			M^{E+1}P^{-rE}.
		\end{equation}
		Thus, on the dyadic shell
		\[
		\mathfrak M_C(P;2M)\setminus\mathfrak M_C(P;M),
		\]
		we obtain
		\begin{equation}
			P^{rE}
			\int_{\mathfrak M_C(P;2M)\setminus\mathfrak M_C(P;M)}
			|\Lambda_{\alpha,P}(f)|\,d\alpha
			\lesssim
			M^{E+1-\kappa_0/E} =  M^{-\gamma}.
			\label{eq:dyadic-shell}
		\end{equation}
		
		Finally, simultaneous Dirichlet approximation gives
		\[
		\mathfrak M_C(P;M_*)=\T^E
		\qquad\text{for}\qquad
		M_*\asymp P^{rE/(E+1)}.
		\]
		Indeed, taking \(H=P^{r/(E+1)}\), one obtains
		\[
		q\leq H^E=M_*,\qquad\|qC_r\alpha\|_\infty
		\leq H^{-1}=M_*P^{-r}.
		\]
		Summing \eqref{eq:dyadic-shell} over
		\(M=M_0,2M_0,\ldots,M_*\) proves
		\eqref{eq:integrated-minor}.
	\end{proof}
	
	We can now estimate the high-frequency error term $\cE_{q,L}^{\eps,\Z}(f)$.
	
	\begin{prop}
		\label{prop:disc-minor}
		Assume \eqref{eq:lattice-dimension}.  For every $\eps>0$, the error term
		in \eqref{eq:discrete-decomp} satisfies
		\begin{equation}\label{eq:disc-minor}
			|\cE_{q,L}^{\eps,\Z}(f)|\leq C\eps^c
		\end{equation}
		uniformly for $|f|\leq1$, provided $q=q(\eps)$, $q\mid L$, and
		$P_0(\eps)\leq P\ll N$, where $P_0(\eps)\leq\eps^{-C}$ and
		$C=C_{n,m}>0$.
		Moreover, the common modulus $q=q(\eps)$ and the lower scale threshold
		$P_0(\eps)$ may be chosen so that
		\begin{equation}\label{eq:minor-effective-constants}
			\log q(\eps)+\log P_0(\eps)\leq C\eps^{-C}.
		\end{equation}
	\end{prop}
	
	\begin{proof}
		Choose $M_0$ to be a sufficiently large fixed power of $\eps^{-1}$.
		By \cref{lem:pruned-lattice}, the contribution outside
		$\mathfrak M_C(P;M_0)$ is $O(\eps^c)$.  Enlarge the common
		denominator $q(\eps)$ by $|\det C_r|$ and by the denominators up to $M_0$,
		so that these vector major arcs are absorbed into the major term of
		\eqref{eq:discrete-decomp}.  Since $P^{rE-D}\sum_v$ in
		\eqref{eq:circle-representation} equals
		$P^{rE}(P^{-D}\sum_v)$, \eqref{eq:integrated-minor} gives
		\eqref{eq:disc-minor} once $P\geq P_0(\eps)$.  The construction of
		$M_0$, $q(\eps)$, and the finitely many small-scale cutoffs gives
		\eqref{eq:minor-effective-constants}.
	\end{proof}
	
	\medskip
	Next, we deal with the sum of the intermediate terms. The key tool we use is Villarroya's transference argument for multi-linear multiplier operators.
	
	\begin{prop}\label{prop:disc-middle}
		For a lacunary sequence of levels $L_j$ and physical scales $P_j=L_j^{1/r}$, $1\leq j\leq J$,
		\begin{equation}\label{eq:disc-middle}
			\sum_j|\cK_{q,L_j}^{\eps,\Z}(f)|^2
			\lesssim_{m,Q,\De} \eps^{-C}q^{D+2E},
		\end{equation}
		with a constant $C=C_{n,m}$, uniformly in the number of levels $J$ and in $|f|\leq1$.
	\end{prop}
	
	\begin{proof}
		Expanding $\chi_q$ in its finite Fourier series on
		$H_{m,\Z}^n/qH_{m,\Z}^n$, each mode merely modulates the external
		functions. Thus it suffices to prove the estimate without $\chi_q$, accepting an extra factor $q^{D+2E}$ via a simple application of Cauchy-Schwarz.
		
		After the same Cauchy--Schwarz step as in \cref{prop:middle} as in the Euclidean case, the sampled
		kernel is the restriction to the lattice of $K_J^\eps$ in
		\eqref{eq:KJ}.  Poisson summation writes its torus multiplier as a sum over
		\[
		h\in (H_{m,\Z}^n)^*\times(H_{m,\Z}^n)^*\simeq\Z^{2D}
		\]
		of the Euclidean multiplier.  Insert fixed smooth cutoffs in every input frequency on a fundamental cube.  The map $L$ in
		\eqref{eq:kernel-frequency-map} is surjective onto the two kernel-frequency spaces, so every fixed $h$ translation can be absorbed into modulations of the external input functions.
		
		For the finitely many bounded vectors $h$, view the scalar $2m$-linear form as
		a $(2m-1)$-linear operator
		\[
		\prod_{i=1}^{2m-1}L^{2m}\longrightarrow
		L^{2m/(2m-1)}.
		\]
		The continuous integral-rank MTT estimate \cite{MTT02} then transfers to the corresponding sequence spaces by Villarroya's multi-linear transference
		theorem \cite[Theorem~3.3]{Villarroya11}.  Note that the proof there is written in one dimension for notational simplicity, but the paper explicitly notes, and it is easy to see from its proof that it
		is unchanged for multi-linear operators on $\R^n$ and $\Z^n$.
		
		For a large $h$, the cutoff multiplier $m_{J,h}$ obeys, for every
		fixed derivative order and every $M$,
		\begin{equation}\label{eq:alias-decay}
			\norm{\partial^\beta m_{J,h}}_\infty
			\leq C_{\beta,M,\eps}(1+|h|)^{-M},
		\end{equation}
		uniformly in $J$.  Sufficiently many derivatives give
		\begin{equation}\label{eq:alias-L1}
			\norm{\check m_{J,h}}_{L^1}
			\leq C_{M,\eps}(1+|h|)^{-M}.
		\end{equation}
		The physical-space kernel here lives on the full product of input
		variables.  The resulting multi-linear norm has the same decay and is summable over $h$.  This proves
		\eqref{eq:disc-middle}.
	\end{proof}
	
	We can now give the proof of \cref{thm:lattice-large scales}.
	
	\begin{proof}[Proof of \cref{thm:lattice-large scales}]
		Let $c(\de)$ be the $q$-independent constant in
		\eqref{eq:disc-main}, and let
		\[
		c_0:=\inf_{0<\eps\leq\eps_0}c_\eps>0
		\]
		for a sufficiently small fixed $\eps_0$.  Choose
		$\eps=\eps(\delta)\leq\eps_0$ so that the bound in
		\eqref{eq:disc-minor} is at most $c_0c(\de)/2$, and then take the
		common denominator $q=q(\eps)$ in \eqref{eq:discrete-decomp}.  The uniform local conditions
		\eqref{eq:uniform-local} and \eqref{eq:uniform-zero} control,
		respectively, the exact arithmetic normalization and the lower
		bound for the main term uniformly over the indicated levels and moduli.
		
		Suppose there are arbitrarily large admissible levels, divisible by $q$,
		for which no configuration exists.  Select a finite lacunary sequence
		$L_1<\cdots<L_J$ of such levels, write $P_j=L_j^{1/r}$, and require
		\[
		P_1\geq\max\{Cq(\eps),P_0(\eps)\}.
		\]
		Here $C$ is the constant in \cref{prop:disc-main}. Choose a box of side
		$N\geq\eta^{-1}P_J$ on which $A$ has density comparable to
		$\delta$.  Such a box may be chosen arbitrarily large by the definition
		of upper Banach density.  Translate the box to $[1,N]^n$ and let $f$ be
		the indicator of the corresponding translate of $A$ inside that box.
		The exact count vanishes at every $L_j$.  By
		\cref{prop:disc-main,prop:disc-minor} and
		\eqref{eq:discrete-forms-decomp},
		\[
		|\cK_{q,L_j}^{\eps,\Z}(f)|
		\geq c_\eps c(\de)-|\cE_{q,L_j}^{\eps,\Z}(f)|
		\geq \frac12c_0c(\de)=:c'(\delta)>0.
		\quad(1\leq j\leq J).
		\]
		Thus the left side of \eqref{eq:disc-middle} is at least
		$J(c'(\delta))^2$, contradicting \cref{prop:disc-middle} for large $J$.
		All sufficiently large admissible levels divisible by $q$ therefore occur.
	\end{proof}

	
	\section{A quantitative finitary variant}
	\label{sec:quantitative-lattice}
	
	We prove a quantitative result about the existence of barycentric configuration in a finite set $A\subs [N]^n$ of density $\de$, which is $Q$-similar to a given simplex $\De_o$. The proof will combine our basic decomposition with a  density increment method. Doing this over a family of scales is technically involved, mainly due to the inverse results for Roth type affine configurations where one does not have good control of the scales at which the density of the underlying set increases. We handle this by assigning a large family scales to each density, which remains closed under the induction on the densities. 
	
	\begin{defn}[$Q$-similar simplex--center copy]
		\label{def:similar-copy}
		Let $L\in\N$.  A collection
		\[
		z,\quad z+v_1,\ldots,z+v_m\in\Z^n
		\]
		is a \emph{$Q$-similar copy of $\De_o$ together with its center, at level
			$L$}, if
		\begin{equation}\label{eq:quant-similar-copy}
			\sum_{i=1}^m v_i=0,
			\qquad
			Q(v_i-v_j)=L\ell_{ij}\quad(i<j).
		\end{equation}
		The corresponding physical scale is $P=L^{1/r}$.  The terminology refers to the edge
		equations in \eqref{eq:quant-similar-copy}; for a general form $Q$ it does not assert that a linear map carries $\De_o$ to
		$(v_1,\ldots,v_m)$.
	\end{defn}

	\begin{thm}\label{thm:quantitative-lattice}
		Assume the hypotheses of \cref{thm:lattice-large scales}.  There is a
		constant $C=C(m,n,Q,\De)$ such that, for every $0<\delta\leq1$ and every scale $R_0\geq1$, there are a finite set of levels
		$\mathcal T_{\delta,R_0}\subset\N$ and an integer
		$N_*(\delta,R_0)$ satisfying
		\begin{equation}\label{eq:quant-N-bound}
			N_*(\delta,R_0)
			\leq
			\exp\!\left(\exp(C\delta^{-C})(1+\log R_0)\right)
			=(eR_0)^{\exp(C\delta^{-C})},
		\end{equation}
		with the following property.  If $N\geq N_*(\delta,R_0)$ and
		\[
		A\subset[1,N]^n\cap\Z^n,
		\qquad |A|\geq\delta N^n,
		\]
		then $A$ contains a $Q$-similar copy of $\De_o$ together with its center at
		some level $L\in\mathcal T_{\delta,R_0}$.  Moreover,
		\begin{equation}\label{eq:quant-level-range}
			R_0\leq L^{1/r}\leq N_*(\delta,R_0)
			\qquad(L\in\mathcal T_{\delta,R_0}).
		\end{equation}
	\end{thm}

	\subsection{A quantitative block dichotomy}
	
	Let us first record the three local inputs.  For functions
	$f_0,\ldots,f_m$ we use the multi-linear versions of the forms introduced
	in Section~\ref{sec:lattice-proof}, replacing the ``diagonal" product by
	$f_0(z)\prod_{i=1}^m f_i(z+v_i)$.
	
	Fix throughout
	\begin{equation}\label{eq:theta-choice}
		\theta:=\frac1{n+1}.
	\end{equation}
	Choose constants $c_*,C_*>0$ once and for all, with $c_*$ sufficiently
	small and $C_*$ sufficiently large, and set
	\begin{equation}\label{eq:kappa-choice}
		\kappa_\rho:=c_*\rho^{C_*}
		\qquad(0<\rho\leq1).
	\end{equation}
	These constants are set to be below the density increments obtained from \cref{lem:local-U2-increment}.
	
	To make the conditioning cells precise, put $M_H:=\lfloor H/q\rfloor$ and,
	for every $a\in(\Z/q\Z)^n$, partition $a+q\Z^n$ into scalar grids
	\[
	C=a+q\bigl(b+[1,M_H]^n\bigr),\qquad b\in M_H\Z^n.
	\]
	Let $\mathcal C_{H,q}$ be the cells contained in $[1,N]^n$ and discard the
	boundary cells, whose union has size $O(HN^{n-1})$.  Write
	\[
	d_A(C):=\frac{|A\cap C|}{|C|},
	\qquad
	h_{H,q}:=\sum_{C\in\mathcal C_{H,q}}d_A(C)\1_C.
	\]
	If $\varphi$ is a fixed nonnegative Schwartz function with
	$\int\varphi=1$ and $\widehat\varphi$ supported in a sufficiently small cube, let
	\begin{equation}\label{eq:grid-smoothing}
		\varphi_{H,q}(x):=q^nH^{-n}\1_{q\Z^n}(x)\varphi(x/H),
		\qquad F_{H,q}:=\1_A*\varphi_{H,q}.
	\end{equation}
	If $T=H/q\geq2$, then
	\[
	\psi_T(y):=\varphi_{H,q}(qy)=T^{-n}\varphi(y/T),
	\qquad \sum_{y\in\Z^n}\psi_T(y)=1,
	\]
	the last identity following from Poisson summation and the Fourier support of
	$\varphi$.  Thus convolution by $\psi_T$ is a positive mean-preserving
	operator on each normalized residue-class grid.  If $f=\1_A$ and
	$f_a(y)=f(a+qy)$, then
	\[
	F_{H,q}(a+qy)=f_a*\psi_T(y),
	\qquad 0\leq F_{H,q}\leq1.
	\]
	Choosing $\varphi$ as a positive majorant and absorbing its rapidly
	decaying tails, one has $F_{H,q}\gtrsim h_{H,q}$ away from the boundary.
	For later use, put
	\[
	\mathcal S_q:=\{u\in H_{m,\Z}^n/qH_{m,\Z}^n:
	\Phi(u)\equiv0\pmod q\}.
	\]
	
	The following lemma is similar in spirit to that of \cite[Lemma 3.1]{MSimplex09}
	\begin{lem}[Asymmetrically smoothed coarse main term]
		\label{lem:quant-coarse-main}
		Let $A\subset[1,N]^n$ have density at least $\rho$, and put
		$t_0=m+1$.  There are effective constants
		$c_\rho^*,C_\rho^*>0$ such that the following holds.
		Let $P=L^{1/r}$, let $q\mid L$, and choose
		\[
		P\geq C_\rho^*q,
		\qquad
		H=C_\rho^*P,
		\qquad H\leq c_\rho^*N.
		\]
		If $d_A(C)\leq\rho+\kappa_\rho$ for every $C\in\mathcal C_{H,q}$, then
		\begin{equation}\label{eq:asymmetric-main-lower}
			\cM_{q,L}^{\Z}
			(\1_A,F_{H,q},\ldots,F_{H,q})
			\geq c\,\beta_q(0)\rho^{t_0},
		\end{equation}
		where $c>0$ depends only on the initial configuration $\De_o$.  Moreover,
		\begin{equation}\label{eq:coarse-effective-constants}
			\log C_\rho^*
			+\log (c_\rho^*)^{-1}
			\leq\exp(C\rho^{-C}).
		\end{equation}
	\end{lem}
	
	\begin{proof}
		For each $u\in\mathcal S_q$, split the base point into residue classes modulo
		$q$.  The shifts in the support of $\omega_P$ have size $O(P)$ and hence
		are $O_\rho(H)$.  The discarded boundary contributes $O(H/N)$, which is
		absorbed by the choice of $c_\rho^*$.  By our assumption, positivity of
		$F_{H,q}$, and the standard Varnavides averaging give a lower bound
		$c\rho^{t_0}$ for \eqref{eq:asymmetric-main-lower} with $h_{H,q}$.  Only the affine relations
		\[
		z,\quad z+v_1,\ldots,z+v_{m-1},
		\quad z-v_1-\cdots-v_{m-1}
		\]
		are used; no independence of the vectors $v_i$ is required.
		
		There are $|\mathcal S_q|=\beta_q(0)q^{D-E}$ vectors $v=(v_1,\ldots,v_m)$ satisfying $\Phi(v)\equiv 0\pmod{q}$.  Their density $q^{-E}$ cancels the factor $q^E$ of $\chi_q$ in
		\eqref{eq:chiq}, while $P^{-D}$ cancels the number of lifts.  This gives
		the factor $\beta_q(0)$ in \eqref{eq:asymmetric-main-lower}; compare
		\eqref{eq:small-frame-count}--\eqref{eq:unweighted-null-count} and
		\cite{MSimplex09}.
	\end{proof}
	
	For a function $G$ on $[1,N]^n$ and $a\in(\Z/q\Z)^n$, set
	$G_a(y):=G(a+qy)$ and extend $G_a$ by zero.  We use the standard local
	$U^2$ norm at scale $S$ and, crucially, retain the average over all residue
	classes:
	\begin{align}
		\|G_a\|_{U^2(S)}^4
		&:=\mathbb E_x\mathbb E_{h,k\in[-S,S]^n}
		G_a(x)\overline{G_a(x+h)}
		\overline{G_a(x+k)}G_a(x+h+k),\nonumber\\
		\mathcal U_{S,q}(G)^4
		&:=\mathbb E_{a\bmod q}\|G_a\|_{U^2(S)}^4.
		\label{eq:averaged-local-U2}
	\end{align}
	Here all averages have the usual normalized finite-box interpretation; changing
	the boundary convention changes the quantities by $O(qS/N)$.
	
	\begin{lem}[Local $U^2$ increment on a scalar grid]
		\label{lem:local-U2-increment}
		Let $A\subset[1,N]^n$ have density
		$\bar\rho:=|A|/N^n\geq\rho$, let $H/q,S\geq2$, and suppose that
		\begin{equation}\label{eq:U2-increment-hypothesis}
			\mathcal U_{S,q}(\1_A-F_{H,q})\geq\eta.
		\end{equation}
		Then there are integers $s,M\geq1$ and a translate $b$ such that
		\[
		b+qs[1,M]^n\subset[1,N]^n,
		\]
		\begin{equation}\label{eq:U2-grid-size}
			s\leq C_\eta S^{1-\theta},\qquad
			M\geq c_\eta S^\theta,\qquad
			sM\leq C_\eta S,
		\end{equation}
		and
		\begin{equation}\label{eq:U2-density-increment}
			\frac{|A\cap(b+qs[1,M]^n)|}{M^n}
			\geq\bar\rho+c\eta^C-C_\eta\frac{H+qS}{N}.
		\end{equation}
		In particular, if $H+qS\leq c_\eta N$, then, after adjusting the
		constants, the right side is at least $\bar\rho+c\eta^C\geq
		\rho+c\eta^C$.
		The constants may be chosen so that
		$\log C_\eta+\log c_\eta^{-1}\ll\eta^{-C}$.
	\end{lem}
	
	\begin{proof}
		Put $f=\1_A$, $T=H/q$, and
		\[
		\Omega_a:=\{y\in\Z^n:a+qy\in[1,N]^n\},
		\qquad U_a:=\|(f-F_{H,q})_a\|_{U^2(S)}.
		\]
		The hypothesis says $\mathbb E_aU_a^4\geq\eta^4$.  Since $U_a\leq1$,
		thresholding shows that $U_a\geq\eta/2$ on a set of residue classes of
		measure $\gg\eta^4$.  Apply the standard Fourier--Dirichlet $U^2$ inverse
		argument with the individual norm $U_a$ in those fibres, and choose any
		admissible grid with the same bounds in the remaining fibres.  Averaging the
		resulting lower bounds (and enlarging the absolute exponent $C$) gives scalar
		grids
		$\mathcal R_a=s_a[1,M_a]^n$ satisfying \eqref{eq:U2-grid-size}.  For these
		grids put
		\[
		\mathcal T_a:=\{t\in\Z^n:\mathcal R_a+t\subset\Omega_a\},
		\qquad
		\mathbb E_{a,t}:=\mathbb E_{a\bmod q}
		\mathbb E_{t\in\mathcal T_a}.
		\]
		Then
		\begin{equation}\label{eq:grid-averages}
			\mathbb E_{a,t}
			\left|
			\mathbb E_{x\in\mathcal R_a+t}
			\bigl(f_a-f_a*\psi_T\bigr)(x)
			\right|
			\geq c\eta^C-O_\eta(qS/N),
		\end{equation}
		The usual proof permits the choices $s_a,M_a$ to depend on $a$; this causes
		no difficulty because the residue average has not been discarded.
		
		Set
		\[
		d_a(t):=\mathbb E_{x\in\mathcal R_a+t}f_a(x)
		\quad(t\in\Z^n),
		\qquad K_Td_a:=d_a*\psi_T.
		\]
		Here $f_a$ is extended by zero, so $d_a$ and hence $K_Td_a$ are defined
		for every $t$, including those near the boundary of $\mathcal T_a$.
		Since convolution commutes with averaging over a translated grid,
		\[
		\mathbb E_{x\in\mathcal R_a+t}(f_a*\psi_T)(x)
		=K_Td_a(t).
		\]
		Thus the inner average in \eqref{eq:grid-averages} is
		$\alpha_a(t):=d_a(t)-K_Td_a(t)$.  Averaging over all residue classes and
		all translates gives
		\[
		\mathbb E_{a,t}d_a(t)=\bar\rho+O_\eta(qS/N),
		\qquad
		\mathbb E_{a,t}K_Tu_a(t)
		=\mathbb E_{a,t}u_a(t)+O_\eta(H/N)
		\quad (|u_a|\leq1).
		\]
		The second identity holds for every bounded family
		$u_a:\Z^n\to\C$ and follows by translating $\mathcal T_a$ under the
		Markov kernel; only a boundary layer of physical width $O(H)$ is lost.
		Consequently \eqref{eq:grid-averages} implies
		\begin{equation}\label{eq:grid-positive-average}
			\mathbb E_{a,t}(\alpha_a(t))_+
			\geq c\eta^C-O_\eta((H+qS)/N).
		\end{equation}
		
		On the other hand, positivity and constant preservation of $K_T$ give the
		pointwise inequality
		\[
		(d_a-K_Td_a)_+
		\leq(d_a-\bar\rho)_+
		+K_T\bigl((\bar\rho-d_a)_+\bigr).
		\]
		If $\Delta:=\sup_{a,t}(d_a(t)-\bar\rho)$, averaging this inequality and
		using the preceding mean identities yields
		\[
		\mathbb E_{a,t}(\alpha_a(t))_+
		\leq2\Delta+O_\eta((H+qS)/N).
		\]
		Together with \eqref{eq:grid-positive-average}, this gives
		$\Delta\geq c\eta^C-C_\eta(H+qS)/N$.  Selecting the corresponding
		$a,t$ and returning to physical coordinates gives
		$b+qs_a[1,M_a]^n$, with $b=a+qt$, and proves
		\eqref{eq:U2-density-increment}.
	\end{proof}
	
	
	The other input is the explicit form of \cref{prop:disc-middle}:
	for every physically lacunary sequence of levels,
	\begin{equation}\label{eq:quant-middle-q}
		\sum_j|\cK_{q,L_j}^{\eps,\Z}(f)|^2
		\lesssim \eps^{-C}q^{D+2E},
	\end{equation}
	uniformly in the number of levels and in $|f|\leq1$.
	
	\begin{prop}[Quantitative block dichotomy]
		\label{prop:quant-block}
		For every $0<\rho\leq1$ there are an integer $q_\rho$, a block length
		$J_\rho$, and effective constants $b_\rho,C_\rho>0$ and
		$0<c_\rho\leq1$ such that
		\begin{equation}\label{eq:block-parameters}
			\log q_\rho\leq\rho^{-C},\qquad
			J_\rho\leq\exp(\rho^{-C}),
		\end{equation}
		and the following holds.  For every $X\geq2$ one can choose levels
		$L_1<\cdots<L_{J_\rho}$, divisible by $q_\rho$, whose physical scales
		$P_j=L_j^{1/r}$ satisfy
		\begin{equation}\label{eq:block-scales}
			b_\rho q_\rho X^{1/\theta}4^{j-1}
			\leq P_j\leq
			2b_\rho q_\rho X^{1/\theta}4^{j-1}.
		\end{equation}
		If $A\subset[1,N]^n$ has density at least $\rho$ and
		\begin{equation}\label{eq:block-fits}
			P_{J_\rho}\leq c_\rho N,
		\end{equation}
		then either
		\begin{enumerate}
			\item $A$ contains a simplex--center copy at one of the levels $L_j$;
			or
			\item for some $j$ there are integers $s,M\geq1$ and a translate $a$
			such that
			\[
			a+q_\rho s[1,M]^n\subset[1,N]^n,
			\]
			\begin{equation}\label{eq:block-increment-size}
				s\leq C_\rho(P_j/q_\rho)^{1-\theta},
				\qquad M\geq X,
			\end{equation}
			and
			\begin{equation}\label{eq:block-increment}
				\frac{|A\cap(a+q_\rho s[1,M]^n)|}{M^n}
				\geq\rho+\kappa_\rho.
			\end{equation}
		\end{enumerate}
		Moreover,
		\begin{equation}\label{eq:block-constant-bound}
			\log b_\rho+\log c_\rho^{-1}+\log C_\rho
			\leq\exp(C\rho^{-C}).
		\end{equation}
	\end{prop}
	
	\begin{proof}
		Take $\eps_\rho=c_0\,\rho^{C_0}$, with $c_0>0$ small and $C_0$ large
		enough that
		\cref{prop:disc-minor} gives
		\begin{equation}\label{eq:quant-minor-small}
			|\cE_{q(\eps_\rho),L}^{\eps_\rho,\Z}(\1_A)|
			\leq \frac{c}{8}\beta_*\rho^{t_0},
			\qquad t_0=m+1.
		\end{equation}
		Fix also
		\begin{equation}\label{eq:block-U2-threshold}
			\eta_\rho:=c_1\rho^{t_0},
		\end{equation}
		where $c_1>0$ is sufficiently small.  The constants in
		\eqref{eq:kappa-choice} are chosen so that $\kappa_\rho$ is below the
		increment $c\eta_\rho^C$ supplied by \cref{lem:local-U2-increment}.
		Choose the common denominator $q_\rho=q(\eps_\rho)$ in
		\eqref{eq:discrete-decomp}.  The construction in
		\cref{prop:disc-minor} gives $\log q_\rho\leq\rho^{-C}$, and its lower
		scale threshold is absorbed into $b_\rho$.
		Choose the fit constant in the proposition so that
		\begin{equation}\label{eq:block-fit-constant}
			c_\rho\leq
			\min\left\{
			\frac{c_\rho^*}{C_\rho^*},
			\frac{c_{\eta_\rho}}{C_\rho^*+1}
			\right\}.
		\end{equation}
		Shrink $c_\rho$ further, if necessary, to meet the upper-scale
		hypotheses in \cref{prop:disc-minor}.  Thus $P_j\leq c_\rho N$
		controls both the $C_\rho^*P_j/N$ and $P_j/N$ boundary errors.
		Choose $b_\rho$ sufficiently large to absorb the lower threshold
		$P_0(\eps_\rho)$, the condition $P_1\geq C_\rho^*q_\rho$, the constants in
		\cref{lem:local-U2-increment}, the nesting inequality
		\eqref{eq:tree-nesting-template} below, and the requirement $M\geq X$.
		
		For each $j$, let $H_j=C_\rho^*P_j$ and $F_j=F_{H_j,q_\rho}$.  If, for
		some $j$, a cell $C\in\mathcal C_{H_j,q_\rho}$ has density at least
		$\rho+\kappa_\rho$, then, since
		$C=b+q_\rho[1,M_{H_j}]^n$, it gives the second alternative with $s=1$.
		Our choice of $b_\rho$ ensures $M_{H_j}\geq X$.
		Otherwise \cref{lem:quant-coarse-main} applies for every $j$.  If the first
		alternative fails, every level $L_j$ is missing, and for each $j$ we write
		\eqref{eq:discrete-forms-decomp} as
		\begin{equation}\label{eq:block-four-terms}
			0=\cB_j+\cD_j+\cK_j+\cE_j,
		\end{equation}
		where
		\begin{align*}
			\cB_j&:=c_{\eps_\rho}\cM_{q_\rho,L_j}^{\Z}
			(\1_A,F_j,\ldots,F_j),\\
			\cD_j&:=c_{\eps_\rho}\left[
			\cM_{q_\rho,L_j}^{\Z}(\1_A,\ldots,\1_A)
			-\cM_{q_\rho,L_j}^{\Z}(\1_A,F_j,\ldots,F_j)
			\right].
		\end{align*}
		Thus
		\begin{equation}\label{eq:block-B-lower}
			\cB_j\geq c\,\beta_{q_\rho}(0)\rho^{t_0}.
		\end{equation}
		
		Telescoping the $m$ replacements in $\cD_j$ and decomposing both the frame
		and the base point into residue classes modulo $q_\rho$, a typical term
		has the normalized form
		\begin{equation}\label{eq:residue-normalized-error}
			\beta_{q_\rho}(0)
			\mathbb E_{u\in\mathcal S_{q_\rho}}
			\mathbb E_{a\bmod q_\rho}
			\Lambda_{u,a}(F_0,\ldots,
			\1_A-F_j,\ldots,F_m),
		\end{equation}
		where $\Lambda_{u,a}$ is an unweighted affine form on a
		$q_\rho$-grid at scale $P_j/q_\rho$.  Thus the height $q_\rho^E$ of
		$\chi_{q_\rho}$ has canceled against the $q_\rho^{-E}$ density of null
		residue frames $v$ satisfying $\Phi(v)\equiv 0\pmod{q_\rho}$.
		
		The affine system satisfying $v_1+\cdots+v_m=0$ has complexity one.  Put
		\[
		S_j:=\left\lfloor P_j/q_\rho\right\rfloor.
		\]
		The localized generalized von Neumann inequality, with all residue-class
		averages in \eqref{eq:residue-normalized-error} retained, therefore gives
		the following estimate.  Indeed, for fixed $u$, translating $a$ by the
		residue of the balanced slot merely permutes the residue classes, so the
		$u$-average does not alter the displayed norm:
		\[
		\mathbb E_{a\bmod q_\rho}
		\| (\1_A-F_j)_a\|_{U^2(S_j)}
		\gtrsim\rho^{t_0}
		\]
		whenever
		\begin{equation}\label{eq:D-large}
			|\cD_j|>\frac{c}{4}\beta_{q_\rho}(0)\rho^{t_0},
		\end{equation}
		By Jensen's inequality this implies
		\[
		\mathcal U_{S_j,q_\rho}(\1_A-F_j)\geq\eta_\rho.
		\]
		Now \cref{lem:local-U2-increment} gives the second alternative.  The
		choice of $c_\rho$ absorbs the error in
		\eqref{eq:U2-density-increment}, while the choice of $b_\rho$ ensures
		$M\geq X$ and gives the bound for $s$ in
		\eqref{eq:block-increment-size}.
		
		If neither alternative occurs, then
		\eqref{eq:block-B-lower}, \eqref{eq:quant-minor-small}, and
		\eqref{eq:block-four-terms} imply
		\[
		|\cK_j|\geq c'\beta_*\rho^{t_0}
		\qquad(1\leq j\leq J_\rho).
		\]
		On the other hand, \eqref{eq:quant-middle-q} gives
		\[
		\sum_{j=1}^{J_\rho}|\cK_j|^2
		\lesssim \eps_\rho^{-C}q_\rho^{D+2E}.
		\]
		It is therefore enough to take
		\begin{equation}\label{eq:block-length-choice}
			J_\rho>
			C\eps_\rho^{-C}q_\rho^{D+2E}
			(c'\beta_*\rho^{t_0})^{-2}.
		\end{equation}
		This choice satisfies $J_\rho\leq\exp(\rho^{-C})$.
		
		Finally, with $Y_j=b_\rho q_\rho X^{1/\theta}4^{j-1}$, choose
		\[
		L_j=q_\rho\left\lceil Y_j^r/q_\rho\right\rceil.
		\]
		With this fixed choice of $b_\rho$, these levels satisfy
		\eqref{eq:block-scales}; in particular their physical scales are
		lacunary.  Since the logarithmic width of the block is $O(J_\rho)$,
		all constants obey \eqref{eq:block-constant-bound}.  This proves the proposition.
	\end{proof}
	
	\subsection{Construction of the finite tree of scales.}
	
	\begin{proof}[Proof of \cref{thm:quantitative-lattice}]
		We first define some fixed parameters independently of the set which will allow ``enough room" to define  a family of scales for our induction on the density arguments. Let $\rho_0=\delta$ and
		define
		\begin{equation}\label{eq:density-thresholds}
			\rho_{h+1}:=\rho_h+\kappa_{\rho_h}.
		\end{equation}
		Let $h_*$ be the first index for which $\rho_{h_*}>1$.  Since
		$\kappa_\rho=c_*\rho^{C_*}$,
		\begin{equation}\label{eq:tree-depth}
			h_*\leq\delta^{-C}.
		\end{equation}
		Write $q_h=q_{\rho_h}$ and $J_h=J_{\rho_h}$.
		
		The constants $b_\rho$ in \cref{prop:quant-block} are
		chosen large enough that, whenever $P$ satisfies the lower bound in
		\eqref{eq:block-scales}, we have 
		\begin{equation}\label{eq:tree-nesting-template}
			q_\rho
			\left\lceil C_\rho(P/q_\rho)^{1-\theta}\right\rceil X
			\leq P,
		\end{equation}
		which follows from $P\geq b_\rho q_\rho X^{1/\theta}$.
		Let
		\begin{equation}\label{eq:tree-B-rho}
			B_\rho:=2c_\rho^{-1}b_\rho q_\rho4^{J_\rho-1}+2.
		\end{equation}
		Then \eqref{eq:block-parameters} and
		\eqref{eq:block-constant-bound} give
		\begin{equation}\label{eq:tree-B-bound}
			\log B_\rho\leq\exp(C\rho^{-C}).
		\end{equation}
		Set
		\begin{equation}\label{eq:tree-terminal}
			X_{h_*}:=\max\{2,R_0\}.
		\end{equation}
		For $h=h_*-1,h_*-2,\ldots,0$, define
		\begin{equation}\label{eq:tree-room-recursion}
			X_h:=B_{\rho_h}X_{h+1}^{1/\theta},
		\end{equation}
		and then apply \cref{prop:quant-block} with $X=X_{h+1}$ to choose a block of lacunary scales
		\[
		L_{h,1}<\cdots<L_{h,J_h},
		\qquad P_{h,j}:=L_{h,j}^{1/r}.
		\]
		The explicit upper bound in \eqref{eq:block-scales} and the definition of
		$B_{\rho_h}$ give
		\begin{equation}\label{eq:tree-block-room}
			P_{h,J_h}\leq c_{\rho_h}X_h,
		\end{equation}
		while \eqref{eq:tree-nesting-template} gives, for
		\begin{equation}\label{eq:tree-s-bound}
			s^{\max}_{h,j}:=
			\left\lceil C_{\rho_h}(P_{h,j}/q_h)^{1-\theta}\right\rceil,
		\end{equation}
		one has the relation
		\begin{equation}\label{eq:tree-nesting}
			q_hs^{\max}_{h,j}X_{h+1}\leq P_{h,j}
			\qquad(0\leq h<h_*,\ 1\leq j\leq J_h).
		\end{equation}
		This is precisely the reason for choosing the parent block at physical
		scale comparable to $q_hX_{h+1}^{1/\theta}$.
		
		We now build the ``decision" tree.  Let $\mathcal V_0$ consist of a
		single root $\varnothing$, and put $g_{\varnothing}=1$.  Having constructed
		the finite set $\mathcal V_h$ of nodes of generation $h$, attach to every
		$\nu\in\mathcal V_h$ one child $(\nu,j,s)$ for every
		\[
		1\leq j\leq J_h,
		\qquad 1\leq s\leq s^{\max}_{h,j},
		\]
		and define
		\begin{equation}\label{eq:tree-dilation}
			g_{(\nu,j,s)}:=g_\nu q_hs.
		\end{equation}
		Translations and side lengths do not label formal children: they affect
		the location and available room of an actual increment box, but not the
		level of a copy lifted back to the original coordinates.
		
		Define the predetermined root-level family
		\begin{equation}\label{eq:tree-family}
			\mathcal T_{\delta,R_0}
			:=
			\left\{
			g_\nu^rL_{h,j}:
			0\leq h<h_*,\ \nu\in\mathcal V_h,\ 1\leq j\leq J_h
			\right\}.
		\end{equation}
		This family is finite and depends only on $\delta,R_0$ and the fixed
		configuration, not on $A$.
		
		Let $A\subset[1,N]^n$ have density at least $\delta$, where $N\geq X_0$,
		and suppose for a contradiction that $A$ avoids every level in
		\eqref{eq:tree-family}.  We follow a path down the already constructed
		tree.  At an actual node $\nu\in\mathcal V_h$ we keep an affine scalar
		embedding
		\[
		\iota_\nu(x)=b_\nu+g_\nu x
		\]
		and a normalized set
		\[
		A_\nu:=\{x\in[1,N_\nu]^n:\iota_\nu(x)\in A\}
		\]
		satisfying
		\begin{equation}\label{eq:tree-node-invariant}
			|A_\nu|\geq\rho_hN_\nu^n,
			\qquad N_\nu\geq X_h.
		\end{equation}
		At the root take $b_{\varnothing}=0$, $g_{\varnothing}=1$, and
		$N_{\varnothing}=N$.
		
		If $A_\nu$ contained a copy at a direct level $L_{h,j}$, homogeneity would
		give
		\[
		Q\bigl(g_\nu(v_i-v_j)\bigr)
		=g_\nu^rL_{h,j}\ell_{ij},
		\]
		so $A$ would contain a copy at the level $g_\nu^rL_{h,j}$ in
		\eqref{eq:tree-family}.  Hence all levels in the block are missing from
		$A_\nu$.  The fit condition follows from
		\eqref{eq:tree-block-room} and \eqref{eq:tree-node-invariant}, so
		\cref{prop:quant-block} supplies, for some $j,s$, a grid
		\[
		a+q_hs[1,M]^n\subset[1,N_\nu]^n,
		\qquad M\geq X_{h+1},
		\]
		on which the normalized density is at least $\rho_{h+1}$.  Select the
		formal child $\nu'=(\nu,j,s)$ and put
		\[
		b_{\nu'}:=b_\nu+g_\nu a,
		\qquad g_{\nu'}:=g_\nu q_hs,
		\qquad N_{\nu'}:=M.
		\]
		The corresponding set $A_{\nu'}$ satisfies
		\eqref{eq:tree-node-invariant} at generation $h+1$.  Notice that the full
		$M$-box is retained; no passage to a subbox of exactly side $X_{h+1}$ is
		needed.
		
		After $h_*$ such forward steps, \eqref{eq:tree-node-invariant} would give a
		set of density at least $\rho_{h_*}>1$, a contradiction.  Thus $A$ contains a
		copy at a level in \eqref{eq:tree-family}.
		
		It remains to bound the scales.  Along a child edge generated by $(j,s)$,
		\eqref{eq:tree-nesting} gives
		\begin{equation}\label{eq:lifted-tree-nesting}
			g_{\nu'}X_{h+1}
			=g_\nu q_hsX_{h+1}
			\leq g_\nu P_{h,j}.
		\end{equation}
		Since \eqref{eq:tree-block-room} also gives $P_{h,j}\leq X_h$, induction
		down the formal tree yields the exact invariant
		\begin{equation}\label{eq:root-scale-invariant}
			g_\nu X_h\leq X_0
			\qquad(\nu\in\mathcal V_h).
		\end{equation}
		Hence every physical test scale $g_\nu P_{h,j}$ in
		\eqref{eq:tree-family} is at most $X_0$.  On the other hand,
		\eqref{eq:tree-nesting} gives
		$P_{h,j}\geq X_{h+1}\geq X_{h_*}\geq R_0$, while $g_\nu\geq1$.
		Therefore $g_\nu P_{h,j}\geq R_0$ for every test level.
		
		Write $x_h=\log X_h$ and $d_\theta=1/\theta=n+1$.  From
		\eqref{eq:tree-room-recursion} and \eqref{eq:tree-B-bound},
		\begin{equation}\label{eq:tree-log-recurrence}
			x_h\leq d_\theta x_{h+1}+\exp(C\rho_h^{-C}).
		\end{equation}
		Since $h_*\leq\delta^{-C}$ and $\rho_h\geq\delta$, iteration gives
		\[
		x_0\leq\exp(C\delta^{-C})(1+\log R_0).
		\]
		Taking $N_*(\delta,R_0)=\lceil X_0\rceil$ proves
		\eqref{eq:quant-N-bound} and \eqref{eq:quant-level-range}, and completes
		the proof of \cref{thm:quantitative-lattice}.  Taking $R_0=1$ proves
		\cref{thm:lattice-similar}.
	\end{proof}
	
	\begin{rem}
		The key observation is that the refinement parameter $s$ can make the tree
		very broad, but it does not enter its depth or the maximum-scale recurrence
		\eqref{eq:tree-log-recurrence}.  The construction of the tree is necessary
		because we do not have a direct quantitative lower bound, independent of the
		modulus $q$, for the main term in the discrete case when $m\geq3$, owing to
		the congruence restriction $\Phi(v)\equiv0\pmod q$.  Such estimates may be
		possible and would make the quantitative argument substantially simpler and
		more similar to the Euclidean case.
		
	\end{rem}

\end{document}